\documentclass{amsart} 
\usepackage[foot]{amsaddr}
\usepackage{amsthm, amsmath, mathtools}
\usepackage{amssymb}
\usepackage{amscd}
\usepackage{enumerate}
\usepackage{mathrsfs} 
\usepackage[curve,matrix,arrow,frame,tips]{xy}
\usepackage{colonequals}
\usepackage{verbatim}
\usepackage{pdfsync}
\usepackage{graphicx}
\usepackage[usenames,dvipsnames]{color}

\usepackage{url}

\usepackage{fullpage}

\usepackage[normalem]{ulem}

\numberwithin{equation}{section} 
\numberwithin{table}{section}

\theoremstyle{plain}
\newtheorem{theorem}{Theorem}[section]
\newtheorem{proposition}[theorem]{Proposition}
\newtheorem{lemma}[theorem]{Lemma}

\theoremstyle{definition}
\newtheorem{definition}[theorem]{Definition}
\newtheorem{example}[theorem]{Example}

\newtheorem{nota}[theorem]{Notation}
\newtheorem{assume}[theorem]{Assumption}

\theoremstyle{remark}
\newtheorem{remark}[theorem]{Remark}

\newcommand{\HH}{\mathrm{H}}

\begin{document} 

\title{Holomorphic differentials of Klein four covers}

\author{Frauke M. Bleher \and Nicholas Camacho}
\address{Department of Mathematics\\University of Iowa\\
14 MacLean Hall\\Iowa City, IA 52242-1419\\ U.S.A.}
\email{frauke-bleher@uiowa.edu, nicholascamacho@live.com}
\thanks{Both authors were supported in part by NSF Grant No.\ DMS-1801328. Frauke M. Bleher is the corresponding author.}

\date{January 14, 2023}


\begin{abstract}
Let $k$ be an algebraically closed field of characteristic two, and let $G$ be isomorphic to $\mathbb{Z}/2\times\mathbb{Z}/2$. Suppose $X$ is a smooth projective irreducible curve over $k$ with a faithful $G$-action, and assume that the cover $X\to X/G$ is totally ramified, in the sense that it is ramified and every branch point is totally ramified. 
We study to what extent the lower ramification groups of the closed points of $X$ 
determine the isomorphism types of the indecomposable $kG$-modules and the multiplicities with which they occur as direct summands of the space $\HH^0(X,\Omega_{X/k})$ of holomorphic differentials of $X$ over $k$. 
In the case when $X/G=\mathbb{P}^1_k$, we completely determine the decomposition of $\HH^0(X,\Omega_{X/k})$ into a direct sum of indecomposable $kG$-modules. 
Moreover, we show that the isomorphism classes of indecomposable $kG$-modules that actually occur as direct summands belong to an infinite list of non-isomorphic indecomposable $kG$-modules that contain modules of arbitrarily large $k$-dimension. In particular, our results show that \cite[Theorem 6.4]{MarquesWard2018} is incorrect.
\end{abstract}

\maketitle


\section{Introduction}
\label{s:intro}

Suppose $k$ is an algebraically closed field, and $X$ is a smooth projective irreducible curve over $k$ on which a finite group $G$ acts faithfully on the right. Let $\Omega_{X/k}$ be the sheaf of relative differentials of $X$ over $k$,
and let $\HH^0(X,\Omega_{X/k})$ be the space of holomorphic differentials of $X$ over $k$. 
It is a classical problem, posed by Hecke in \cite{Hecke1928}, to determine the indecomposable $kG$-modules that occur as direct summands of $\HH^0(X,\Omega_{X/k})$ together with their multiplicities. If the characteristic of $k$ does not divide the order of $G$, this was solved by Chevalley and Weil  in \cite{ChevalleyWeil1934}.

For the remainder of the paper, we assume that the characteristic of $k$ is a prime $p$ dividing $\#G$. Many authors have studied this problem and have made some progress by either making assumptions on the group $G$ or the ramification of the cover $X\to X/G$. See \cite{ValentiniMadan1981,Kani1986,Nakajima1986,RzedowskiCVillaSMadan1996,Kock2004,KaranikolopoulosKontogeorgis2013,MarquesWard2018,BleherChinburgKontogeorgis2020,Garnek2022} for a sample of previous results. We will now briefly discuss two of these articles, as they are closely related to our results.

The authors of \cite{RzedowskiCVillaSMadan1996} studied the case when $G$ is an elementary abelian $p$-group and $X/G=\mathbb{P}^1_k$. They made additional assumptions on the ramification of the $G$-cover $X \to \mathbb{P}^1_k$ that ensured that the indecomposable $kG$-modules that occur as direct summands of $\HH^0(X,\Omega_{X/k})$ belong, up to isomorphism, to a list of $|G| - 1$ non-isomorphic indecomposable $kG$-modules that have pairwise distinct $k$-dimensions between $1$ and $|G| - 1$. 

The authors of \cite{MarquesWard2018} considered the case when $G$ is a solvable group, $X/G=\mathbb{P}^1_k$, and the $G$-cover $X \to \mathbb{P}^1_k$ decomposes into a tower of Artin-Schreier extensions and/or Kummer extensions that are in so-called ``global standard form." Moreover, if $G$ is abelian with Sylow $p$-subgroup $P$, they claimed in \cite[Theorem 6.4]{MarquesWard2018} that only a list of $(|P| - 1)\cdot |G/P|$ non-isomorphic indecomposable $kG$-modules can occur as direct summands of $\HH^0(X,\Omega_{X/k})$ and that these $kG$-modules have $k$-dimensions between $1$ and $|P| - 1$. However, this statement is incorrect, 
as we will show below.

The goal of this paper is to study the smallest case when the set of isomorphism classes of finitely generated indecomposable $kG$-modules is infinite, which occurs when the characteristic of $k$ is two and $G$ is a Klein four group, i.e. $G$ is isomorphic to $\mathbb{Z}/2\times \mathbb{Z}/2$. Note that 
the isomorphism classes of finitely generated indecomposable $kG$-modules can still be classified, using a certain one-parameter family in every even $k$-dimension (see \S\ref{s:klein4rep} for a  description). 
We assume that the Klein four cover $\pi:X\to X/G =: Y$ is totally ramified, in the sense that it is ramified and every ramification point in $X$ has $G$ as its inertia group. We study the following two questions:
\begin{enumerate}
\item To what extent do the lower ramification groups of the closed points of $X$ determine the precise $kG$-module structure of $\HH^0(X,\Omega_{X/k})$, in the sense that they give the precise decomposition of $\HH^0(X,\Omega_{X/k})$ into a direct sum of indecomposable $kG$-modules?
\item What can we say about the list of isomorphism classes of indecomposable $kG$-modules that actually occur as direct summands of $\HH^0(X,\Omega_{X/k})$?
\end{enumerate}

Our main idea is to define infinitely many operators given by certain elements in the group ring $kG$ (see Notation \ref{not:operators}) and to compare the filtrations defined by these on both the $\mathcal{O}_Y$-$G$-sheaf $\pi_*\Omega_{X/k}$ and the $kG$-module $\HH^0(X,\Omega_{X/k})$ (see Propositions \ref{prop:crucial} and \ref{prop:dimension}). We then use this relationship, which is governed by the ramification data of the cover $\pi$, to determine information about the $kG$-module structure of $\HH^0(X,\Omega_{X/k})$. When the genus of $X/G$ is arbitrary, we obtain restrictions on the indecomposable $kG$-modules and their multiplicities that can occur as direct summands of $\HH^0(X,\Omega_{X/k})$ (see Theorem \ref{thm:mostgeneral}). Moreover, we obtain sufficient conditions that lead 
to the following complete answers to Questions (1) and (2) above (see Theorem \ref{thm:onlyB1orB2B3small} and Examples \ref{ex:oops1} and \ref{ex:oops2} for more details):

\begin{theorem}
\label{thm:awkward1}
Suppose $\pi:X\to Y=X/G$ is a totally ramified Klein four cover. 
Moreover, assume that for each branch point $y\in Y$ of $\pi$, there exists a tower of function fields $k(Y)\subset k(Z_y) \subset k(X)$ such that $k(Z_y) = k(Y)(u_y)$ and $k(X)=k(Z_y)(w_y)$ are degree $2$ Artin-Schreier extensions of $k(Y)$ and $k(Z_y)$, respectively, satisfying the following two conditions:
\begin{itemize}
\item[(i)] If $z_y\in Z_y$ lies above $y$ and $x_y\in X$ lies above $z_y$, then $j_y:=-\mathrm{ord}_{z_y}(u_y)$ and $J_y:=-\mathrm{ord}_{x_y}(w_y)$ are positive odd integers such that $j_y\le J_y$ and the lower ramification groups at $x_y$ inside $G$ have jumps at $j_y$ and $J_y$, and
\item[(ii)] either $j_y=J_y$ and $(g-1)(w_y)\in k^\times$ for all non-identity elements $g \in G$, or $j_y<J_y$ and $j_y=1$.
\end{itemize}
Then the lower ramification groups of the closed points of $X$ that ramify in the cover $\pi$ 
fully determine the $kG$-module structure of $\HH^0(X,\Omega_{X/k})$. 

Moreover, the list of isomorphism classes of  indecomposable $kG$-modules that actually occur as direct summands of $\HH^0(X,\Omega_{X/k})$ for such $X$ is infinite.
\end{theorem}

Theorem \ref{thm:awkward1} includes the case in \cite{RzedowskiCVillaSMadan1996} when $p=2=n$ and extends it from $Y=\mathbb{P}^1_k$ to an arbitrary smooth projective curve $Y$ (see Example \ref{ex:Nickthesis}). As a consequence of Theorem \ref{thm:awkward1} (and, more generally, Theorem \ref{thm:onlyB1orB2B3small}), we will show in Examples \ref{ex:oops1} and \ref{ex:oops2} that \cite[Theorem 6.4]{MarquesWard2018} is not correct; see also Remark \ref{rem:wrongMarquesWard}.

Note that condition (i) in Theorem \ref{thm:awkward1} is satisfied for every branch point $y\in Y$ (see Remark \ref{rem:twocases}), but condition (ii) may not be. In Theorem \ref{thm:onlyB1orB2B3small}, we will extend condition (ii) to include one additional case, which is a bit more technical to state but for which we can also fully determine the $kG$-module structure of $\HH^0(X,\Omega_{X/k})$ without any assumptions on the genus of $X/G$ (see Remark \ref{rem:awkward1follows} and condition (b) in Theorem \ref{thm:onlyB1orB2B3small}). We will show in Examples \ref{ex:oops1} and \ref{ex:oops2}  that the list of isomorphism classes of  indecomposable $kG$-modules that occur as direct summands of $\HH^0(X,\Omega_{X/k})$ in the situation of Theorem \ref{thm:awkward1} (and, more generally, Theorem \ref{thm:onlyB1orB2B3small}) for various $X$ is infinite, and we write down the precise list. Moreover, we give examples in which $\HH^0(X,\Omega_{X/k})$ is a direct sum of an arbitrary finite number of non-isomorphic indecomposable $kG$-modules.

In the case when $X/G=\mathbb{P}^1_k$, we obtain the following complete answers to Questions (1) and (2) above, without any additional assumptions on the ramification behavior (see Theorem \ref{thm:P1main} and Example \ref{ex:KatzGabber} for more details):

\begin{theorem}
\label{thm:awkward2}
Suppose $\pi:X\to \mathbb{P}^1_k=X/G$ is a totally ramified Klein four $G$-cover. Then the lower ramification groups of the closed points of $X$ that ramify in the cover $\pi$ fully determine the $kG$-module structure of $\HH^0(X,\Omega_{X/k})$. More precisely, the isomorphism classes of the indecomposable $kG$-modules that actually occur as direct summands and their multiplicities in $\HH^0(X,\Omega_{X/k})$ can be described explicitly. Furthermore, the list of these isomorphism classes is infinite and contains indecomposable $kG$-modules of arbitrarily large finite $k$-dimension.
\end{theorem}

It is an interesting question whether the set of isomorphism classes of indecomposable modules that arises in Theorem \ref{thm:awkward2} changes when one drops the assumption that $X/G$ is isomorphic to $\mathbb{P}^1_k$. We will give a sufficient criterion for this set to not change (see Remarks \ref{rem:candomore} and \ref{rem:sodumber!}).

\medskip

The paper is organized as follows. In \S\ref{s:klein4}, we relate filtrations of the $\mathcal{O}_Y$-$G$-sheaf $\pi_*\Omega_{X/k}$ to filtrations of the $kG$-module $\HH^0(X,\Omega_{X/k})$ for arbitrary Klein four covers (see Propositions \ref{prop:crucial} and \ref{prop:dimension}). We also provide the detailed calculations needed for totally ramified Klein four covers in Propositions \ref{prop:totallyramified} and \ref{prop:dims}. We  use this in Theorem \ref{thm:mostgeneral} to obtain restrictions on the indecomposable $kG$-modules and their multiplicities that can occur as direct summands of $\HH^0(X,\Omega_{X/k})$. In Theorem \ref{thm:onlyB1orB2B3small}, we give sufficient conditions under which Theorem \ref{thm:mostgeneral} leads to a full determination of the $kG$-module structure of $\HH^0(X,\Omega_{X/k})$. Moreover, we show in Examples \ref{ex:oops1} and \ref{ex:oops2} that every member of the full list of possible indecomposable $kG$-modules in Theorem \ref{thm:mostgeneral} actually occurs as a direct summand of $\HH^0(X,\Omega_{X/k})$ for various $X$. 
In \S\ref{s:KleinfourP1}, we focus on the case when $X/G=\mathbb{P}^1_k$. We determine in Theorem \ref{thm:P1main} the full $kG$-module structure of $\HH^0(X,\Omega_{X/k})$ in this case. Moreover, we show in Example \ref{ex:KatzGabber} that every member of the full list of possible indecomposable $kG$-modules in Theorem \ref{thm:P1main} actually occurs as a direct summand of $\HH^0(X,\Omega_{X/k})$ for various $X$. 

Part of this paper is the Ph.D. thesis of the second author under the supervision of the first (see \cite{CamachoThesis2021}). We would like to thank the referee for helpful comments and suggestions.


\section{Klein four covers in characteristic two}
\label{s:klein4}

Throughout this paper, we make the following assumptions.

\begin{assume}
\label{ass:general}
Let $k$ be an algebraically closed field of characteristic 2, and let $\pi:X \to Y$ be a Galois cover of smooth projective curves over $k$ with Galois group $G$ isomorphic to $\mathbb{Z}/2\times \mathbb{Z}/2$. Let $L=k(X)$ and $K=k(Y)$. Write $L = K(u,v)$ where 
\begin{equation}
\label{eq:field}
u^2-u=p, \quad v^2-v=q
\end{equation}
for certain $p,q\in K$ such that none of $p,q,p-q$ is of the form $s^2-s$ for any $s\in K$.
Write $G=\{1,\sigma,\tau,\sigma\tau\} = \langle \sigma, \tau\rangle$ such that
\begin{equation}
\label{eq:Gaction}
\sigma(u)=u,\tau(u)=u+1 \quad \mbox{and} \quad \sigma(v)=v+1,\tau(v)=v.
\end{equation}
\end{assume}

We have three intermediate fields between $K$ and $L$, corresponding to three smooth projective curves $Z_0$, $Z_1$, $Z_2$ over $k$ that are Galois covers of $Y$ of degree 2. Their function fields are given by
\begin{equation}
\label{eq:3curves}
k(Z_0)=K(u),\quad k(Z_1)=K(v),\quad k(Z_2)=K(u+v).
\end{equation}
We have
$$k(Z_0)=L^{\langle \sigma\rangle},\quad k(Z_1)=L^{\langle \tau\rangle},\quad k(Z_2)=L^{\langle \sigma\tau\rangle}.$$

We need the following notation.
\begin{nota}
\label{not:operators}
For each closed point $(a:b)\in\mathbb{P}^1_k$, define the following elements in the group ring $kG$:
\begin{eqnarray*}
f_0&:=&1,\\
f_1&:=&\sigma-1\quad \mbox{and} \quad f_1'\;:=\;\tau-1,\\
f_2=f_2^{(a:b)}&:=&
\left\{\begin{array}{cl}
(\frac{a}{b})(\sigma-1)+(\tau-1) & \mbox{ if } b\ne 0,\\
(\sigma-1) & \mbox{ if } b=0,
\end{array}\right.\\
f_3&:=&(\sigma-1)(\tau-1),\\
f_4&:=&0.
\end{eqnarray*}
We view these elements as operators acting on $k(X)$, on $kG$-modules, and on $\pi_*\mathcal{F}$ for certain coherent $\mathcal{O}_X$-$G$-modules $\mathcal{F}$ (see \cite[Section 2]{BleherChinburgKontogeorgis2020} for more details). For $i\in\{0,2,3,4\}$ a superscript $(i)$ denotes the kernel of the action of $f_i$, whereas a superscript $(1)$ denotes the intersection of the kernels of the actions of $f_1$ and $f_1'$. Since the kernel of the action of $f_2=f_2^{(a:b)}$ may be different for different choices of $(a:b)$ in $\mathbb{P}^1_k$, we often add $(a:b)$ to the superscript $(2)$.
\end{nota}

For every $(a:b)\in \mathbb{P}^1_k$, we obtain a filtration 
$$0=kG^{(0)} \subset kG^{(1)} \subset kG^{(2),(a:b)} \subset kG^{(3)} \subset kG^{(4)}=kG.$$
Notice that $kG^{(1)}$ is the socle of $kG$ and $kG^{(3)}$ is the radical of $kG$, and that both $kG^{(1)}$ and $kG/kG^{(3)}$ have $k$-dimension one. Moreover, $kG^{(3)}/kG^{(1)}$ has $k$-dimension two, and all its proper non-trivial $k$-subspaces are of the form $kG^{(2),(a:b)}/kG^{(1)}$ for some closed point $(a:b)\in \mathbb{P}^1_k$. Hence, as $(a:b)$ ranges over all closed points in $\mathbb{P}^1_k$, these filtrations range over all possible composition series of $kG$. Note that for different $(a:b)\in \mathbb{P}^1_k$, the $kG$-modules $kG^{(2),(a:b)}$ are not isomorphic.

The main idea of this paper is to relate, for all $(a:b)\in \mathbb{P}^1_k$, the filtration of the $\mathcal{O}_Y$-$G$-sheaf $\pi_*\Omega_{X/k}$ given by $\{(\pi_*\Omega_{X/k})^{(i)}\}_{i=0}^4$ to the filtration of the $kG$-module $\HH^0(X,\Omega_{X/k})$ given by $\{\HH^0(X,\Omega_{X/k})^{(i)}\}_{i=0}^4$ and to use this relationship to determine as much as possible of the $kG$-module structure of $\HH^0(X,\Omega_{X/k})$.

The following remark is important for this comparison.

\begin{remark}
\label{rem:important}
By the normal basis theorem, $k(X)$ is a free rank one module for the group ring $k(Y)G$. Fix a closed point $(a:b)\in \mathbb{P}^1_k$, and define $\widetilde{f}_i:=f_i$ for $i\in\{0,3\}$, $\widetilde{f}_2:=f_2^{(a:b)}$, and $\widetilde{f}_1:=f_1$ if $b\ne 0$ and $\widetilde{f}_1:=f_1'$ if $b=0$. Then, for $0\le i\le 3$, the map
\begin{equation}
\label{eq:map!}
\widetilde{f}_i: \quad \frac{k(X)^{(i+1)}}{k(X)^{(i)}} \to k(Y)
\end{equation}
given by multiplication by $\widetilde{f}_i$ is a $k(Y)G$-module isomorphism. Let $h_1,h_2,h_3,h_4\in k(X)$ be such that
\begin{equation}
\label{eq:basiscondition}
h_{i+1}\in k(X)^{(i+1)} \quad\mbox{and}\quad \widetilde{f}_i\,h_{i+1}\in k^\times\quad\mbox{for $0\le i\le 3$.}
\end{equation}
Then $\{h_1,\ldots,h_{i+1}\}$ is a $k(Y)$-basis of $k(X)^{(i+1)}$ and $\mathrm{ord}_y(\widetilde{f}_i\,h_{i+1})=0$ for all points $y\in Y$. 
Notice that the latter condition is used when adapting the proof of \cite[Proposition 4.1]{BleherChinburgKontogeorgis2020} to our situation (see Proposition \ref{prop:crucial} below).
\end{remark}

The next result is proved using similar arguments as in the proof of \cite[Proposition 4.1]{BleherChinburgKontogeorgis2020}; we will therefore omit its proof.

\begin{proposition}
\label{prop:crucial}
Under Assumption $\ref{ass:general}$, let $\mathcal{D}_{X/Y}^{-1}$ be the inverse different of $X$ over $Y$, and fix a closed point $(a:b)\in \mathbb{P}^1_k$. For $0\le i\le 3$, there exist divisors $D_i$ on $Y$, which are uniquely determined by $\pi_*\mathcal{D}_{X/Y}^{-1}$, such that there are isomorphisms of $\mathcal{O}_Y$-$G$-modules
$$(\pi_*\Omega_{X/k})^{(i+1)}/(\pi_*\Omega_{X/k})^{(i)} \cong \Omega_{Y/k}\otimes_{\mathcal{O}_Y}\mathcal{O}_Y(D_i).$$
More precisely, write $D_i=\sum_{y\in Y} d_{y,i} y$. For each $y\in Y$, let $h_1=h_1(y)$, $h_2 = h_2(y)$, $h_3 = h_3(y)$, and $h_4 = h_4(y)$ be elements of $k(X)$ satisfying $(\ref{eq:basiscondition})$ from Remark $\ref{rem:important}$. Then
$$d_{y,i} = - \mathrm{min}\{\mathrm{ord}_y(\lambda_{i+1}) \; : \; \lambda_1h_1 +\cdots+ \lambda_{i+1} h_{i+1} \in (\pi_*\mathcal{D}_{X/Y}^{-1})_y \mbox{ for some } \lambda_1,\ldots,\lambda_{i+1} \in k(Y)\}. $$
Moreover, we have the following for $x\in X$ and $y=\pi(x)$:
\begin{itemize}
\item[(i)] If $x$ is not a ramification point of $\pi$ then $d_{y,i}=0$.
\item[(ii)] If $x$ is totally ramified and if $\mathrm{ord}_x(h_j(y))$ are pairwise distinct modulo $4$ for $j\in\{1,2,3,4\}$ then 
\begin{equation}
\label{eq:divisor}
d_{y,i}=\left\lfloor \frac{d_{x/y} +\mathrm{ord}_x(h_{i+1}(y))}{4}\right\rfloor
\end{equation}
where $d_{x/y}$ is the different exponent of $x$ over $y$.
\end{itemize}
\end{proposition}

The next result is proved using similar arguments as in the proof of \cite[Lemma 4.2]{BleherChinburgKontogeorgis2020}; we will therefore omit its proof.

\begin{proposition}
\label{prop:dimension}
Under Assumption $\ref{ass:general}$, fix a closed point $(a:b)\in \mathbb{P}^1_k$, and let $D_0,\ldots,D_3$ be the divisors from Proposition $\ref{prop:crucial}$. Suppose $D_3=0$ and $\mathrm{deg}(D_i)>0$ for $0\le i \le 2$.
Then, for $0\le i\le 3$, there are isomorphisms of $kG$-modules
$$\HH^0(X,\Omega_{X/k})^{(i+1)}/\HH^0(X,\Omega_{X/k})^{(i)} \cong\HH^0(Y,\Omega_{Y/k}\otimes_{\mathcal{O}_Y}\mathcal{O}_Y(D_i)).$$
\end{proposition}

We now make the following assumption.

\begin{assume}
\label{ass:totallyramified}
The cover $\pi:X\to Y$ in Assumption \ref{ass:general} is totally ramified, in the sense that $\pi$ is ramified and every branch point $y\in Y$ of $\pi$ is totally ramified. In other words, if $y\in Y$ is a branch point of the cover $\pi$ then there exist $s_{p,y},s_{q,y},s_{p+q,y}\in K=k(Y)$ such that each of 
$$\mathrm{ord}_y(p-(s_{p,y}^2-s_{p,y})), \quad \mathrm{ord}_y(q-(s_{q,y}^2-s_{q,y})) , \quad \mathrm{ord}_y(p+q-(s_{p+q,y}^2-s_{p+q,y}))$$ 
is negative and odd, where $p$ and $q$ are as in (\ref{eq:field}) (see, for example, \cite[Proposition 3.7.8]{Stichtenoth2009}).
\end{assume}

This assumption will enable us to find for each branch point $y\in Y$, elements $h_1(y),\ldots,h_4(y)$ of $k(X)$ satisfying $(\ref{eq:basiscondition})$ from Remark \ref{rem:important} such that $\mathrm{ord}_x(h_j(y))$ are pairwise distinct modulo $4$ for $j\in\{1,2,3,4\}$ (see Proposition \ref{prop:totallyramified} below). Moreover, the formula in (\ref{eq:divisor}) shows then that $D_3=0$ and $\mathrm{deg}(D_i)>0$ for $0\le i \le 2$. In other words, Proposition \ref{prop:dimension} applies. 

The following remark and lemma describe the basic cases we need to consider for each branch point $y \in Y$.

\begin{remark}
\label{rem:twocases}
Under Assumptions $\ref{ass:general}$ and $\ref{ass:totallyramified}$, let $y\in Y$ be a branch point, and let $x\in X$ be the unique point above it. Following the discussion in \cite[Section 3]{WuScheidler2010}, and in particular using \cite[Equations (3.6) and (3.7) and Theorem 3.11]{WuScheidler2010}, we see the following. There exist positive odd integers $m_y\le M_y$ such that the lower ramification groups at $x$ are given as 
$$\qquad \qquad G=G_{x,0}=G_{x,1}=\cdots=G_{x,m_y} > G_{x,m_y+1}=\cdots = G_{x,m_y+2(M_y-m_y)}>G_{x,m_y+2(M_y-m_y)+1}=1.$$
Moreover (using, for example, \cite[Proposition 3.7.8]{Stichtenoth2009}) there exist $s_{p,y},s_{q,y},s_{p+q,y}\in K=k(Y)$ such that either 
\begin{itemize}
\item[(i)] $m_y=M_y$ and $\mathrm{ord}_y(p-(s_{p,y}^2-s_{p,y}))=-m_y=\mathrm{ord}_y(q-(s_{q,y}^2-s_{q,y})) =\mathrm{ord}_y(p+q-(s_{p+q,y}^2-s_{p+q,y}))$, or 
\item[(ii)] $m_y< M_y$ and either
	\begin{itemize}
	\item[(a)] $\mathrm{ord}_y(p-(s_{p,y}^2-s_{p,y}))=-m_y$, $\mathrm{ord}_y(q-(s_{q,y}^2-s_{q,y}))=-M_y=\mathrm{ord}_y(p+q-(s_{p+q,y}^2-s_{p+q,y}))$, or 
	\item[(b)] $\mathrm{ord}_y(q-(s_{q,y}^2-s_{q,y}))=-m_y$, $\mathrm{ord}_y(p-(s_{p,y}^2-s_{p,y}))=-M_y=\mathrm{ord}_y(p+q-(s_{p+q,y}^2-s_{p+q,y}))$, or 
	\item[(c)] $\mathrm{ord}_y(p+q-(s_{p+q,y}^2-s_{p+q,y}))=-m_y$, $\mathrm{ord}_y(p-(s_{p,y}^2-s_{p,y}))=-M_y=\mathrm{ord}_y(q-(s_{q,y}^2-s_{q,y}))$.
	\end{itemize}
\end{itemize}
In the cases (i) or (ii)(a), define $u_y:=u-s_{p,y}$, $p_y:=p-(s_{p,y}^2-s_{p,y})$ and $v_y:=v-s_{q,y}$, $q_y :=q-(s_{q,y}^2-s_{q,y})$, in the case (ii)(b), define $u_y:=v-s_{q,y}$, $p_y:=q-(s_{q,y}^2-s_{q,y})$ and $v_y:=u-s_{p,y}$, $q_y :=p-(s_{p,y}^2-s_{p,y})$, and in the case (ii)(c), define $u_y:=u+v-s_{p+q,y}$, $p_y:=p+q-(s_{p+q,y}^2-s_{p+q,y})$ and $v_y:=v-s_{q,y}$, $q_y :=q-(s_{q,y}^2-s_{q,y})$. In particular, this means
$$\mathrm{ord}_x(u_y)=-2m_y,\quad \mathrm{ord}_x(v_y)=-2M_y.$$
By \cite[Corollary 3.10 and the proof of Theorem 3.11]{WuScheidler2010}, there exist $\alpha_y , \beta_y \in K$ such that $w_y = v_y + \alpha_y + \beta_y u_y$ is an Artin-Schreier generator of $L=k(X)$ over $K(u_y)=k(Y)(u_y)$ with the following properties:
$$\mathrm{ord}_x(w_y)=-m_y-2(M_y-m_y),\quad\mathrm{ord}_x(\beta_y)=-2(M_y-m_y).$$
Moreover,
$$\begin{array}{lll}
(\sigma -1)(w_y) = 1, & (\tau-1)(w_y) = \beta_y &\mbox{ in the situation of (i) or (ii)(a),}\\
(\tau-1)(w_y) = 1, & (\sigma -1)(w_y) = \beta_y &\mbox{ in the situation of (ii)(b), and}\\
(\sigma\circ\tau -1)(w_y) = 1, & (\tau -1)(w_y) = \beta_y &\mbox{ in the situation of (ii)(c).}
\end{array}$$
\end{remark}

The following equation gives a useful connection between upper bounds of running indices occurring in the next result, see in particular equations (\ref{eq:needref1}) and (\ref{eq:needref2}):
\begin{equation}
\label{eq:duh1new}
\left\lfloor \frac{2m_y+3}{4}\right\rfloor - \left\lfloor \frac{m_y+3}{4}\right\rfloor 
= \left\{\begin{array}{ll} \left\lfloor \frac{m_y}{4}\right\rfloor  & \mbox{if } m_y\equiv 1 \mod 4, \mbox{ and}\\
\left\lfloor \frac{m_y}{4}\right\rfloor +1 & \mbox{if } m_y\equiv 3 \mod 4.
\end{array}\right.
\end{equation}

\begin{lemma}
\label{lem:twocases}
Assume the notation from Remark $\ref{rem:twocases}$. Then $\mathrm{ord}_y(\alpha_y)\ge (-M_y+1)/2$. Let $\pi_y$ be a uniformizer at $y$ and write $\alpha_y,\beta_y,p_y,q_y$ as Laurent series in $k[[\pi_y,\pi_y^{-1}]]$ as follows:
$\alpha_y=\pi_y^{(-M_y+1)/2}\sum_{i\ge 0} a_{y,i}\pi_y^i$, $\beta_y=\pi_y^{-(M_y-m_y)/2}\sum_{i\ge 0} b_{y,i}\pi_y^i$,
$p_y=\pi_y^{-m_y}\sum_{i\ge 0} p_{y,i}\pi_y^i$, $q_y=\pi_y^{-M_y}\sum_{i\ge 0} q_{y,i}\pi_y^i$, where $b_{y,0},p_{y,0},q_{y,0}\in k^\times$. Then 
\begin{equation}
\label{eq:needref1}
q_{y,2j} + \sum_{i_1+i_2=j} p_{y,2i_1}b_{y,i_2}^2 = 0 \qquad\mbox{for } 0\le j \le \textstyle{\left\lfloor \frac{m_y}{4}\right\rfloor},
\end{equation} 
and $b_{y,0}\in k-\{0,1\}$ when $m_y=M_y$. Moreover, we obtain
\begin{equation}
\label{eq:needref2}
a_{y,j}^2=q_{y,2j+1} + \sum_{i_1+i_2=j} p_{y,2i_1+1}b_{y,i_2}^2\qquad\mbox{for } 0\le j \le \textstyle{\left\lfloor \frac{2m_y+3}{4}\right\rfloor} - \textstyle{\left\lfloor \frac{m_y+3}{4}\right\rfloor} - 1.
\end{equation}
Define 
\begin{equation}
\label{eq:wuj1}
\widetilde{\alpha}_y:=\pi_y^{(-M_y+1)/2} \left(\sum_{i = 0}^{\left\lfloor \frac{2m_y+3}{4}\right\rfloor - \left\lfloor \frac{m_y+3}{4}\right\rfloor - 1} a_{y,i}\pi_y^i\right)
\end{equation}
and, for $j\ge 0$, define
\begin{eqnarray}
\label{eq:wuj21}
\beta_y(j) &:=& \pi_y^{-(M_y-m_y)/2} \left(\sum_{i=0}^j b_{y,i}\pi_y^i\right) \quad \mbox{and}\\
\label{eq:wuj22}
w_y(j) &:=& v_y + \widetilde{\alpha}_y + \beta_y(j)\,u_y.
\end{eqnarray}
Then
\begin{equation}
\label{eq:useful1}
\mathrm{ord}_x(w_y(j)) \ge -2M_y+4j+4 \qquad\mbox{for $0\le j \le \textstyle{\left\lfloor \frac{m_y}{4}\right\rfloor} -1 $,}
\end{equation}
and 
\begin{equation}
\label{eq:useful2}
\mathrm{ord}_x(w_y(j)) = -m_y-2(M_y-m_y)\qquad\mbox{for $j\ge \textstyle{\left\lfloor \frac{m_y}{4}\right\rfloor}$.}
\end{equation}
\end{lemma}

\begin{proof}
We have
\begin{eqnarray*}
w_y^2-w_y &=& v_y^2-v_y + \beta_y^2 (u_y^2-u_y) + \alpha_y^2-\alpha_y + (\beta_y^2-\beta_y)u_y\\
&=& q_y +\beta_y^2 p_y + \alpha_y^2-\alpha_y + (\beta_y^2-\beta_y)u_y.
\end{eqnarray*}
Let $Z \in\{Z_0,Z_1,Z_2\}$ be the degree two cover of $Y$ such that $u_y\in k(Z)$, and let $z\in Z$ be below $x$ and above $y$. Note that
$$\mathrm{ord}_x(w_y) = \mathrm{ord}_z (w_y^2-w_y),$$
$\mathrm{ord}_z(q_y +\beta_y^2 p_y)\ge -2M_y$ is even, $\mathrm{ord}_z(\alpha_y^2-\alpha_y)$ is either negative and divisible by 4 or non-negative and even, and $\mathrm{ord}_z(u_y) = -m_y$. Hence, to obtain $\mathrm{ord}_z (w_y^2-w_y)=-m_y - 2(M_y-m_y)$, we must satisfy the following two conditions:
\begin{equation}
\label{eq:cond1}
\mathrm{ord}_z(\beta_y^2-\beta_y)=-2(M_y-m_y)\quad\mbox{and}
\end{equation} 
\begin{equation}
\label{eq:cond2}
\mathrm{ord}_z(q_y +\beta_y^2 p_y + (\alpha_y^2-\alpha_y))\ge -m_y - 2(M_y-m_y)+1.
\end{equation}
Condition (\ref{eq:cond1}) implies that $b_{y,0}\ne 1$ in the case when $m_y=M_y$. Considering condition (\ref{eq:cond2}), suppose $\mathrm{ord}_y(\alpha_y)\le (-M_y-1)/2$. Then $\mathrm{ord}_z(\alpha_y^2-\alpha_y))\le -2M_y-2 < \mathrm{ord}_z(q_y +\beta_y^2 p_y)$, which results in a contradiction of condition (\ref{eq:cond2}). Therefore, $\mathrm{ord}_y(\alpha_y)\ge (-M_y+1)/2$. To analyze condition (\ref{eq:cond2}) further, we use the Laurent series expansions of $\alpha_y,\beta_y,p_y,q_y$ in $k[[\pi_y,\pi_y^{-1}]]$. We obtain
\begin{eqnarray*}
q_y +\beta_y^2 p_y  &=& \sum_{i\ge 0} \left( q_{y,i} + \sum_{i_1+2i_2=i} p_{y,i_1}b_{y,i_2}^2 \right) \pi_y^{-M_y+i} \\
&=& \sum_{j=0}^{\lfloor \frac{m_y}{4}\rfloor }  \left( q_{y,2j} + \sum_{i_1+i_2=j} p_{y,2i_1}b_{y,i_2}^2 \right) \pi_y^{-M_y+2j}\\
&&+ \sum_{j=0}^{\lfloor \frac{m_y}{4}\rfloor}  \left( q_{y,2j+1} + \sum_{i_1+i_2=j} p_{y,2i_1+1}b_{y,i_2}^2 \right) \pi_y^{-M_y+2j+1}\\
&&+ \sum_{i\ge 2 \lfloor \frac{m_y}{4}\rfloor + 2} \left( q_{y,i} + \sum_{i_1+2i_2=i} p_{y,i_1}b_{y,i_2}^2 \right) \pi_y^{-M_y+i} 
\end{eqnarray*}
where 
$$\mathrm{ord}_z(\pi_y^{-M_y+2 \lfloor \frac{m_y}{4}\rfloor + 2}) = -2M_y+4 \textstyle{\left\lfloor \frac{m_y}{4}\right\rfloor} + 4
\ge -2M_y+(m_y-3)+4 =-m_y -2(M_y-m_y)+1.$$
For $0\le j \le \lfloor\frac{m_y}{4}\rfloor$, we have that $\mathrm{ord}_z(\pi_y^{-M_y+2j}) = -2M_y+4j$ is even but not divisible by 4,
whereas $\mathrm{ord}_z(\pi_y^{-M_y+2j+1}) = -2M_y+4j+2$ is divisible by 4. Since $\mathrm{ord}_z(\alpha_y)$ and $\mathrm{ord}_z(\alpha_y^2)$ are divisible by 2 and 4, respectively, and since $\mathrm{ord}_z(\pi_y^{(-M_y+1)/2}) = -M_y+1 \ge -m_y -2(M_y-m_y)+1$, we moreover obtain equations (\ref{eq:needref1}) and (\ref{eq:needref2}).
Note that if $m_y\equiv 3\mod 4$ then $\lfloor\frac{2m_y+3}{4}\rfloor - \lfloor\frac{m_y+3}{4}\rfloor - 1 =\lfloor\frac{m_y}{4}\rfloor $. On the other hand, if $m_y\equiv 1\mod 4$ then 
$$\mathrm{ord}_z(\pi_y^{-M_y+2\lfloor\frac{m_y}{4}\rfloor +1}) = -2M_y+4 \textstyle{\left\lfloor\frac{m_y}{4}\right\rfloor}+ 2
=-2M_y + (m_y-1)+2 = -m_y -2(M_y-m_y)+1$$
which means that equation (\ref{eq:needref2}) only follows for $0\le j \le \lfloor\frac{m_y}{4}\rfloor - 1 = \lfloor\frac{2m_y+3}{4}\rfloor - \lfloor\frac{m_y+3}{4}\rfloor - 1$.

Fix now $j\ge 0$ and consider $\widetilde{\alpha}_y$, $\beta_y(j)$ and $w_y(j)$ as in (\ref{eq:wuj1}), (\ref{eq:wuj21}) and (\ref{eq:wuj22}). We have
\begin{eqnarray*}
w_y(j)^2-w_y(j) &=& q_y +\beta_y(j)^2 p_y + \widetilde{\alpha}_y^2-\widetilde{\alpha}_y + (\beta_y(j)^2-\beta_y(j))u_y.
\end{eqnarray*}
Since $b_{y,0}\not\in\{0,1\}$ when $m_y=M_y$, we have $\mathrm{ord}_z(\beta_y(j)^2-\beta_y(j))=-2(M_y-m_y)$. Moreover, equations (\ref{eq:needref1}) and (\ref{eq:needref2}) show that for $j\le \lfloor \frac{m_y}{4}\rfloor -1$,
$$\mathrm{ord}_z (q_y +\beta_y(j)^2 p_y + \widetilde{\alpha}_y^2-\widetilde{\alpha}_y)\ge \mathrm{ord}_z(\pi_y^{-M_y+2j+2})=-2M_y+4j+4.$$
Since
$$-2M_y+4\left(\textstyle{\left\lfloor \frac{m_y}{4}\right\rfloor}-1\right) +4
=\left\{\begin{array}{ll}
-m_y -2(M_y-m_y)-1 & \mbox{ if } m_y\equiv 1\mod 4,\\
-m_y -2(M_y-m_y)-3 & \mbox{ if } m_y\equiv 3\mod 4,
\end{array}\right.$$
we obtain equations (\ref{eq:useful1}) and (\ref{eq:useful2}).
\end{proof}

\begin{remark}
\label{rem:betterwu}
Assume the notation from Remark \ref{rem:twocases} and Lemma \ref{lem:twocases}. Using (\ref{eq:wuj21}) and (\ref{eq:wuj22}), define
\begin{equation}
\label{eq:betterwu}
\widetilde{\beta}_y := \beta_y\left( \textstyle{\left\lfloor\frac{m_y}{4}\right\rfloor} \right) \quad\mbox{and}\quad
\widetilde{w}_y:=w_y\left(\textstyle{\left\lfloor\frac{m_y}{4}\right\rfloor} \right).
\end{equation}
We have
$$\mathrm{ord}_x(\widetilde{w}_y)=-m_y-2(M_y-m_y)\quad\mbox{and}\quad \mathrm{ord}_x(\widetilde{\beta}_y)=-2(M_y-m_y).$$
Moreover,
$$\begin{array}{lll}
(\sigma -1)(\widetilde{w}_y) = 1, & (\tau-1)(\widetilde{w}_y) = \widetilde{\beta}_y &\mbox{ in the situation of Remark \ref{rem:twocases}(i) or (ii)(a),}\\
(\tau-1)(\widetilde{w}_y) = 1, & (\sigma -1)(\widetilde{w}_y) = \widetilde{\beta}_y &\mbox{ in the situation of Remark \ref{rem:twocases}(ii)(b), and}\\
(\sigma\circ\tau -1)(\widetilde{w}_y) = 1, & (\tau -1)(\widetilde{w}_y) = \widetilde{\beta}_y &\mbox{ in the situation of Remark \ref{rem:twocases}(ii)(c).}
\end{array}$$
In other words, if we replace $\alpha_y$ by $\widetilde{\alpha}_y$, $\beta_y$ by $\widetilde{\beta}_y$ and $w_y$ by $\widetilde{w}_y$, then $\widetilde{w}_y=v_y+\widetilde{\alpha}_y+\widetilde{\beta}_y\,u_y$ is an Artin-Schreier generator of $L=k(X)$ over $K(u_y)=k(Y)(u_y)$ satisfying the analogous properties that $w_y$ satsifies in Remark \ref{rem:twocases}. The advantage of using $\widetilde{\alpha}_y$ and $\widetilde{\beta}_y$ instead of $\alpha_y$ and $\beta_y$ is that they are Laurent polynomials rather than Laurent series and that they are fully determined by equations (\ref{eq:needref1}) and (\ref{eq:needref2}).
\end{remark}

We use the following additional notation, based on Remark \ref{rem:twocases} and Lemma \ref{lem:twocases}.

\begin{nota}
\label{not:twocases}
Assume the notation from Remark \ref{rem:twocases} and Lemma \ref{lem:twocases}. 
\begin{itemize}
\item[(i)] If we are in the situation of Remark \ref{rem:twocases}(i), we define $\lambda_y:=b_{y,0}$.  By Lemma \ref{lem:twocases}, $\lambda_y\in k-\{0,1\}$. Additionally, if $1 \le \mathrm{ord}_y(\beta_y-\lambda_y)\le \left\lfloor \frac{m_y}{4}\right\rfloor$, we define $\delta_y:=\mathrm{ord}_y(\beta_y-\lambda_y)$, and otherwise we define $\delta_y:=0$. In particular,  $\delta_y\in\{0,1,\ldots,\left\lfloor \frac{m_y}{4}\right\rfloor\}$. Moreover, if $\delta_y=0$ then $\widetilde{\beta}_y=\lambda_y$ and $\widetilde{w}_y=v_y+\widetilde{\alpha}_y+\lambda_y\, u_y$.
\item[(ii)] If we are in the situation of Remark \ref{rem:twocases}(ii), we define $\delta_y:=-1$. Moreover, we define 
$$\lambda_y:=\left\{\begin{array}{cl} \infty & \mbox{if we are in the situation of Remark \ref{rem:twocases}(ii)(a)},\\
0 & \mbox{if we are in the situation of Remark \ref{rem:twocases}(ii)(b)},\\
1 & \mbox{if we are in the situation of Remark \ref{rem:twocases}(ii)(c)}.
\end{array}\right.$$
\end{itemize}
\end{nota}

The next remark gives a different way of classifying the branch points $y\in Y$ for which $\delta_y=0$ or $\delta_y=-1$.

\begin{remark}
\label{rem:awkward1follows}
Under Assumptions $\ref{ass:general}$ and $\ref{ass:totallyramified}$, let $y\in Y$ be a branch point. 
Suppose $k(Y)\subset k(Z_y) \subset k(X)$ is a tower of function fields such that $k(Z_y) = k(Y)(u_y)$ and $k(X)=k(Z_y)(w_y)$ are degree $2$ Artin-Schreier extensions of $k(Y)$ and $k(Z_y)$, respectively, satisfying the following condition:

If $z_y\in Z_y$ lies above $y$ and $x_y\in X$ lies above $z_y$, then $j_y:=-\mathrm{ord}_{z_y}(u_y)$ and $J_y:=-\mathrm{ord}_{x_y}(w_y)$ are positive odd integers such that $j_y\le J_y$ and the lower ramification groups at $x_y$ inside $G$ have jumps at $j_y$ and $J_y$.

By Remark \ref{rem:twocases}, $u_y$ and $w_y$ always exist, and $j_y=m_y$ and $J_y=m_y+2(M_y-m_y)$. Moreover, we can always arrange that there exist non-identity elements $\sigma_y,\tau_y\in G$ such that $(\sigma_y-1)(u_y)=0$ and $(\tau_y-1)(u_y)=1$. Additionally, we can find $v_y\in k(X)$ such that $k(Y)(v_y)$ is a degree $2$ Artin-Schreier extension of $k(Y)$ and $(\sigma_y-1)(v_y)=1$ and $(\tau_y-1)(v_y)=0$. This implies that $k(X)=k(Y)(u_y,v_y)$ and that there exist $\alpha_y,\beta_y,\gamma_y \in k(Y)$, $\gamma_y\ne 0$, such that $w_y=\alpha_y + \beta_y u_y + \gamma_y v_y$. 

Suppose now additionally that $j_y=J_y$ and that $(g-1)(w_y)\in k^\times$ for all non-identity elements $g \in G$. This implies that $\beta_y=(\sigma_y-1)(w_y)$, $\gamma_y=(\tau_y-1)(w_y)$ and $\beta_y+\gamma_y=(\sigma_y\tau_y-1)(w_y)$ all lie in $k^\times$. In other words, $\beta_y,\gamma_y\in k^\times$ and $\beta_y\ne \gamma_y$. Replacing $w_y$ by $\gamma_y^{-1}w_y$, $\alpha_y$ by $\gamma_y^{-1}\alpha_y$ and $\beta_y$ by $\gamma_y^{-1}\beta_y$, we can assume that $w_y=v_y+\alpha_y + \beta_y u_y$ where $\alpha_y\in k(Y)$ and $\beta_y\in k-\{0,1\}$. In particular, $y$ is as in the situation of Remark \ref{rem:twocases}(i) and $\delta_y=0$. Conversely, if $y$ is as in the situation of Remark \ref{rem:twocases}(i) and $\delta_y=0$, then we can replace $w_y$ from Remark \ref{rem:twocases} by $\widetilde{w}_y$ from Remark \ref{rem:betterwu}, if necessary, to obtain that $(g-1)(w_y)\in k^\times$ for all non-identity elements $g \in G$. 

On the other hand, if $j_y<J_y$, then $y$ is as in the situation of Remark \ref{rem:twocases}(ii) and hence  $\delta_y=-1$.
\end{remark}

For a branch point $y\in Y$ and a point $(a:b)\in \mathbb{P}^1_k$, we define
\begin{equation}
\label{eq:pointswitch}
(c:d):=\left\{\begin{array}{cl} (a:b)&\mbox{if $y$ is as in Remark \ref{rem:twocases}(i) or (ii)(a)},\\
(b:a)&\mbox{if $y$ is as in Remark \ref{rem:twocases}(ii)(b)},\\
(a:a+b)&\mbox{if $y$ is as in Remark \ref{rem:twocases}(ii)(c)}.
\end{array}\right.
\end{equation}

\begin{proposition}
\label{prop:totallyramified}
Under Assumptions $\ref{ass:general}$ and $\ref{ass:totallyramified}$, let $y\in Y$ be a branch point, and let $x\in X$ be a point above $y$. Fix a point $(a:b)\in \mathbb{P}^1_k$ and let $(c:d)$ be as in $(\ref{eq:pointswitch})$.
Let $\widetilde{\beta}_y$ and $\widetilde{w}_y$ be as in $(\ref{eq:betterwu})$, and define
\begin{eqnarray*}
h_1(y)&:=&1,\\
h_2(y)&:=&\left\{\begin{array}{ll}\frac{1}{b}((c+d\widetilde{\beta}_y)u_y+d\widetilde{w}_y)&\mbox{if } b\ne 0 \\ \frac{1}{a}((c+d\widetilde{\beta}_y)u_y+d\widetilde{w}_y)&\mbox{if } b=0\end{array}\right\},\\
h_3(y)&:=&\left\{\begin{array}{ll}u_y&\mbox{if } \delta_y=0 \mbox{ and } c=d\lambda_y\\ \frac{b}{c+d\widetilde{\beta}_y}\widetilde{w}_y &\mbox{if }(\delta_y\ne 0 \mbox{ or } c\ne d\lambda_y)\mbox{ and }b\ne 0\\  \frac{a}{c+d\widetilde{\beta}_y}\widetilde{w}_y &\mbox{if }(\delta_y\ne 0 \mbox{ or } c\ne d\lambda_y)\mbox{ and }b= 0\end{array}\right\},\\
h_4(y)&:=&\widetilde{w}_yu_y.
\end{eqnarray*}
Then $h_1=h_1(y)$, $h_2=h_2(y)$, $h_3=h_3(y)$, and $h_4=h_4(y)$ are elements of $k(X)$ satisfying $(\ref{eq:basiscondition})$ from Remark $\ref{rem:important}$, and $\mathrm{ord}_x(h_j(y))$ are pairwise distinct modulo $4$ for $j\in\{1,2,3,4\}$. Moreover, the divisors $D_0,\ldots,D_3$ from Proposition $\ref{prop:crucial}$ satisfy $D_3=0$ and $\mathrm{deg}(D_i)>0$ for $0\le i \le 2$. More precisely, for $0\le i \le 2$, the coefficient $d_{y,i}$ of $y$ in $D_i$ is given as follows:
\begin{eqnarray*}
d_{y,0}&=&\left\lfloor \frac{3m_y+3}{4}\right\rfloor + \frac{M_y-m_y}{2},\\[2ex]
d_{y,1}&=&\left\{ 
\begin{array}{ll} 
\left\lfloor \frac{2m_y+3}{4}\right\rfloor &\mbox{if } \delta_y=0 \mbox{ and } c=d\lambda_y,\\
\left\lfloor \frac{m_y+3}{4}\right\rfloor + \delta_y &\mbox{if } \delta_y>0 \mbox{ and } c=d\lambda_y,\\
\left\lfloor \frac{m_y+3}{4}\right\rfloor + \frac{M_y-m_y}{2}&\mbox{if } \delta_y=-1 \mbox{ and }d=0,\\
\left\lfloor \frac{m_y+3}{4}\right\rfloor &\mbox{otherwise},
\end{array}\right.\\[2ex]
d_{y,2}&=&\left\{ 
\begin{array}{ll} 
\left\lfloor \frac{m_y+3}{4}\right\rfloor &\mbox{if } \delta_y=0 \mbox{ and } c=d\lambda_y,\\
\left\lfloor \frac{2m_y+3}{4}\right\rfloor - \delta_y &\mbox{if } \delta_y>0 \mbox{ and } c=d\lambda_y,\\
\left\lfloor \frac{2m_y+3}{4}\right\rfloor &\mbox{if } \delta_y=-1 \mbox{ and }d=0,\\
\left\lfloor \frac{2m_y+3}{4}\right\rfloor + \frac{M_y-m_y}{2} & \mbox{otherwise}.
\end{array}\right.
\end{eqnarray*}
\end{proposition}

\begin{proof}
A straightforward computation shows that $\widetilde{f}_j h_j=0$ for $1\le j\le 4$, and that
$\widetilde{f}_i h_{i+1}=1\in k^\times$ for $0\le i\le 3$. Moreover, we have
\begin{eqnarray*}
\mathrm{ord}_x(h_1) &=& 0,\\
\mathrm{ord}_x(h_2) &=&\left\{ 
\begin{array}{ll} 
-m_y &\mbox{if } \delta_y=0 \mbox{ and } c=d\lambda_y,\\
- 2m_y + 4\delta_y  &\mbox{if } \delta_y>0 \mbox{ and } c=d\lambda_y,\\
-2m_y&\mbox{if } \delta_y=-1 \mbox{ and }d=0,\\
-2m_y - 2 (M_y-m_y)&\mbox{otherwise},
\end{array}\right.\\[2ex]
\mathrm{ord}_x(h_3)&=&\left\{ 
\begin{array}{ll} 
-2m_y  &\mbox{if } \delta_y=0 \mbox{ and } c=d\lambda_y,\\
- m_y - 4\delta_y  &\mbox{if } \delta_y>0 \mbox{ and } c=d\lambda_y,\\
- m_y - 2 (M_y-m_y) &\mbox{if } \delta_y=-1 \mbox{ and }d=0,\\
- m_y& \mbox{otherwise},
\end{array}\right.\\
\mathrm{ord}_x(h_4) &=& -3m_y-2(M_y-m_y).
\end{eqnarray*}
Since $M_y-m_y$ is even, this shows that $\mathrm{ord}_x(h_j)$ are pairwise distinct modulo $4$ for $j\in\{1,2,3,4\}$. By Remark \ref{rem:twocases}, the different exponent $d_{x/y} = \sum_{i\ge 0} \#(G_{x,i}-1)$ satisfies
$$d_{x/y} = 3(m_y+1) + 2 (M_y-m_y).$$
Therefore, the formulas for $d_{y,i}$, $0\le i\le 3$, given in Proposition \ref{prop:totallyramified} now follow from Proposition \ref{prop:crucial}(ii). In particular, we obtain $D_3=0$ and $\mathrm{deg}(D_i)>0$ for $0\le i\le 2.$
\end{proof}

\begin{nota}
\label{not:branchpoints}
Define $Y_{\mathrm{br}}\subset Y$ to be the set of branch points of the cover $\pi:X\to Y$, and define $\Lambda_{\mathrm{br}}:=\{\lambda_y\;:\;y\in Y_{\mathrm{br}}\}$. We organize $Y_{\mathrm{br}}$  into 3 disjoint subsets:
\begin{eqnarray*}
B_1 &:=& \{ y\in Y_{\mathrm{br}} \; : \; \delta_y=0\},\\
B_2 &:=& \{ y\in Y_{\mathrm{br}} \; : \; \delta_y \ge 1\},\\
B_3 &:=& \{ y\in Y_{\mathrm{br}} \; : \; \delta_y=-1\}.
\end{eqnarray*}
In other words, $y\in B_1\cup B_2$ is as in the situation of Remark \ref{rem:twocases}(i), and $y\in B_3$ is as in the situation of Remark \ref{rem:twocases}(ii).
Moreover, for $i\in\{1,2,3\}$ and $\lambda\in k\cup\{\infty\}$, define
$$B_{i,\lambda} := \{ y\in B_i \;:\; \lambda_y=\lambda\}.$$
\end{nota}

Using Riemann-Roch, the next result is a direct consequence of Propositions \ref{prop:dimension} and \ref{prop:totallyramified}.

\begin{proposition}
\label{prop:dims}
Under Assumptions $\ref{ass:general}$ and $\ref{ass:totallyramified}$, fix a point $(a:b)\in \mathbb{P}^1_k$. Define $\lambda:=a/b\in k\cup\{\infty\}$, and let $M=\HH^0(X,\Omega_{X/k})$.
For $0\le i\le 3$, the non-negative integer
$$r_i:=\mathrm{dim}_k (M^{(i+1)}/M^{(i)})-g(Y)+1$$
is given as follows:
\begin{eqnarray*}
r_0 &=& \sum_{y\in Y_{\mathrm{br}}} \left( \left\lfloor \frac{3m_y+3}{4}\right\rfloor + \frac{M_y-m_y}{2}\right),\\[2ex]
r_1 &=& \left\{ \begin{array}{cl}
\sum_{y\in B_{1,\lambda}}\left\lfloor \frac{2m_y+3}{4}\right\rfloor +
\sum_{y\in B_{2, \lambda}}\left(\left\lfloor \frac{m_y+3}{4}\right\rfloor+\delta_y\right) \\[.5ex]
 +\sum_{y\in Y_{\mathrm{br}}, \lambda_y\ne \lambda} \left\lfloor \frac{m_y+3}{4}\right\rfloor 
&\mbox{if $\lambda\not\in\{0,1,\infty\}$}, 
\\[2ex]
\sum_{y\in B_{3,\lambda}} \left(\left\lfloor \frac{m_y+3}{4}\right\rfloor+\frac{M_y-m_y}{2}\right) +
\sum_{y\in Y_{\mathrm{br}}-B_{3,\lambda}} \left\lfloor \frac{m_y+3}{4}\right\rfloor & \mbox{if $\lambda\in\{0,1,\infty\}$}, 
\end{array}\right.\\[2ex]
r_2 &=& \left\{ \begin{array}{cl}
\sum_{y\in B_{1,\lambda}}\left\lfloor \frac{m_y+3}{4}\right\rfloor +
\sum_{y\in B_{2,\lambda}}\left(\left\lfloor \frac{2m_y+3}{4}\right\rfloor-\delta_y\right) \\[.5ex]
 +\sum_{y\in Y_{\mathrm{br}}, \lambda_y\ne \lambda} \left(\left\lfloor \frac{2m_y+3}{4}\right\rfloor +\frac{M_y-m_y}{2}\right) 
&\mbox{if $\lambda\not\in\{0,1,\infty\}$}, 
\\[2ex]
 \sum_{y\in B_{3,\lambda}} \left\lfloor \frac{2m_y+3}{4}\right\rfloor +
\sum_{y\in Y_{\mathrm{br}}-B_{3,\lambda}} \left(\left\lfloor \frac{2m_y+3}{4}\right\rfloor +\frac{M_y-m_y}{2}\right)
&\mbox{if $\lambda\in\{0,1,\infty\}$}, 
\end{array}\right.\\[2ex]
r_3&=&1.
\end{eqnarray*}
\end{proposition}

We now use Proposition \ref{prop:dims} 
to determine as much as possible about the indecomposable direct $kG$-module summands of $\HH^0(X,\Omega_{X/k})$. We make the following definition, using Notation \ref{not:indecomposables} for the indecomposable $kG$-modules (see also Remark \ref{rem:indecomposables}).

\begin{definition}
\label{def:decompose}
Let $M=\HH^0(X,\Omega_{X/k})$ and write
$$M=\bigoplus_{h=1}^{\varepsilon_1} N_{2\ell_h,\lambda_h} \;\oplus\; \bigoplus_{i=1}^{\varepsilon_2} M_{2n_i+1,1} \;\oplus\;  \bigoplus_{j=1}^{\varepsilon_3} M_{2n'_j+1,2} \;\oplus\; k^{\oplus \varepsilon_4} \;\oplus\; kG^{\oplus \varepsilon_5}$$
where $\varepsilon_1,\varepsilon_2,\varepsilon_3,\varepsilon_4,\varepsilon_5\in \mathbb{Z}^+\cup\{0\}$, $\ell_h,n_i,n'_j\in \mathbb{Z}^+$, $\lambda_h\in k\cup\{\infty\}$ for $1\le h\le \varepsilon_1, 1\le i\le \varepsilon_2, 1\le j \le \varepsilon_3$. Define $\Lambda_M:=\{\lambda_h\;:\; 1\le h\le \varepsilon_1\}$, and for $\lambda\in k\cup\{\infty\}$, define $\varepsilon_{1,\lambda}:=\#\{h \; :\; 1\le h \le \varepsilon_1, \lambda_h=\lambda\}$.
\end{definition}

\begin{theorem}
\label{thm:mostgeneral}
Under Assumptions $\ref{ass:general}$ and $\ref{ass:totallyramified}$,  let $M=\HH^0(X,\Omega_{X/k})$ be as in Definition $\ref{def:decompose}$. 
Then
\begin{itemize}
\item[(i)] $\Lambda_M \subseteq \Lambda_{\mathrm{br}}$ and, if $\lambda\in \Lambda_M$, then 
$$\varepsilon_{1,\lambda} = 
\left\{\begin{array}{ll}
\sum_{y\in B_{1,\lambda}}\left(\left\lfloor \frac{2m_y+3}{4}\right\rfloor - \left\lfloor \frac{m_y+3}{4}\right\rfloor \right) + \sum_{y\in B_{2,\lambda}}\delta_y & \mbox{if } \lambda\not\in\{0,1,\infty\},\\[1ex]
\sum_{y\in B_{3,\lambda}} \frac{M_y-m_y}{2} & \mbox{if } \lambda\in\{0,1,\infty\},
\end{array}\right.$$
\item[(ii)] $\varepsilon_2= \displaystyle\left(\sum_{y\in Y_{\mathrm{br}}} \left\lfloor \frac{m_y+3}{4}\right\rfloor\right)-1$,
\item[(iii)] $\varepsilon_3+\varepsilon_4 = \displaystyle \sum_{y\in Y_{\mathrm{br}}} \left(\left\lfloor \frac{3m_y+3}{4}\right\rfloor - \left\lfloor \frac{2m_y+3}{4}\right\rfloor\right)$,
\item[(iv)] $\varepsilon_5=g(Y)$.
\end{itemize}
For any $\lambda\in\Lambda_{\mathrm{br}}-\Lambda_M$, we have that
$B_{2,\lambda}\cup B_{3,\lambda}=\emptyset$ and that the points 
$y\in B_{1,\lambda}$
must all satisfy $m_y=1$. 
\end{theorem}

\begin{proof}
Let $M=\HH^0(X,\Omega_{X/k})$ be as in Definition \ref{def:decompose}, and fix $(a:b)\in\mathbb{P}^1_k$. We see from Proposition \ref{prop:dims} and the second row of Table \ref{tab:succdims} that
$$1=r_3=\mathrm{dim}_k (M^{(4)}/M^{(3)})-g(Y)+1 = \varepsilon_5-g(Y)+1$$
which shows part (iv) of Theorem \ref{thm:mostgeneral}. Using this and the remainder of Table \ref{tab:succdims}, we obtain the following values of $r_i = \mathrm{dim}_k (M^{(i+1)}/M^{(i)})-\epsilon_5+1$ for $0\le i\le 2$:
\begin{eqnarray*}
r_0 &=& \sum_{h=1}^{\varepsilon_1}  \ell_h + \sum_{i=1}^{\varepsilon_2} n_i +\sum_{j=1}^{\varepsilon_3} (n'_j+1) + \varepsilon_4 + 1 ,\\[2ex]
r_1 &=& \varepsilon_{1,\lambda} + \varepsilon_2 + 1 
,\\[2ex]
r_2 &=& \displaystyle -\varepsilon_{1,\lambda} + \sum_{h=1}^{\varepsilon_1} \ell_h +  \sum_{i=1}^{\varepsilon_2} n_i +\sum_{j=1}^{\varepsilon_3} n'_j + 1, 
\end{eqnarray*}
where $\lambda=a/b\in k\cup\{\infty\}$.

If 
$\lambda=a/b$ does not lie in $\Lambda_{\mathrm{br}}\cup\Lambda_M$
then the equations for $r_0-r_2$ above and in Proposition \ref{prop:dims} show part (iii) of Theorem \ref{thm:mostgeneral}.

If 
$\lambda=a/b$ does not lie in $\Lambda_{\mathrm{br}}\cup\Lambda_M$
then the equations for $r_1$ above and in Proposition \ref{prop:dims} show 
\begin{equation}
\label{eq:oy1}
\varepsilon_2 + 1 = \sum_{y\in Y_{\mathrm{br}}} \left\lfloor\frac{m_y+3}{4}\right\rfloor
\end{equation}
proving part (ii) of Theorem \ref{thm:mostgeneral}. 

Using (\ref{eq:oy1}) and the equation for $r_1$ in Proposition \ref{prop:dims}, we then obtain for 
$\lambda=a/b\not\in\{0,1,\infty\}$
that
\begin{equation}
\label{eq:oy2}
\varepsilon_{1,\lambda} = \sum_{y\in B_{1,\lambda}}\left(\left\lfloor \frac{2m_y+3}{4}\right\rfloor - \left\lfloor \frac{m_y+3}{4}\right\rfloor \right) + \sum_{y\in B_{2,\lambda}}\delta_y.
\end{equation}
If 
$\lambda=a/b\in\{0,1,\infty\}$
then (\ref{eq:oy1}) and the equation for $r_1$ in Proposition \ref{prop:dims} imply that
\begin{equation}
\label{eq:oy3}
\varepsilon_{1,\lambda} = \sum_{y\in B_{3,\lambda}}\frac{M_y-m_y}{2} .
\end{equation}
In particular, (\ref{eq:oy2}) and (\ref{eq:oy3}) imply that $\Lambda_M\subseteq \Lambda_{\mathrm{br}}$. Moreover, if $\lambda\in\Lambda_{\mathrm{br}}-\Lambda_M$, then we see that
$B_{2,\lambda}\cup B_{3,\lambda}=\emptyset$, and the points $y\in B_{1,\lambda}$ must all satisfy $m_y=1$. This proves part (i) and the last statement of Theorem \ref{thm:mostgeneral}.
\end{proof}

In general, Theorem \ref{thm:mostgeneral} is not sufficient to determine the precise decomposition of $M=\HH^0(X,\Omega_{X/k})$ into a direct sum of indecomposable $kG$-modules, since it does not pin down the values of $\ell_h,n_i,n'_j$ for $1\le h \le \varepsilon_1,1\le i\le \varepsilon_2,1\le j\le \varepsilon_3$, in the notation of Definition \ref{def:decompose}. In the next result, we discuss several cases in which these values can be determined; see also Remark \ref{rem:candomore} for some further discussion.

\begin{theorem}
\label{thm:onlyB1orB2B3small}
Under Assumptions $\ref{ass:general}$ and $\ref{ass:totallyramified}$, let $M=\HH^0(X,\Omega_{X/k})$ be as in Definition $\ref{def:decompose}$, and suppose 
that for all $y\in Y_{\mathrm{br}}$, one of the following conditions holds:
\begin{itemize}
\item[(a)] $y\in B_1$, or
\item[(b)] $y\in B_2$ and $\delta_y=\left\lfloor \frac{2m_y+3}{4}\right\rfloor - \left\lfloor \frac{m_y+3}{4}\right\rfloor$, or
\item[(c)] $y\in B_{3,\lambda}$, for $\lambda\in\{0,1,\infty\}$, and $m_y=1$.
\end{itemize}
Then $\ell_h=1=n_i$ for all $1\le h \le \varepsilon_1,1\le i\le \varepsilon_2$, and $\varepsilon_3=0$. We obtain 
$$M\cong \bigoplus_{\lambda\in \Lambda_{\mathrm{br}}} N_{2,\lambda}^{\oplus a_\lambda}\; \oplus\; M_{3,1}^{\oplus b} \;\oplus\; k^{\oplus c} \;\oplus\; kG^{\oplus d}$$
where 
$$a_\lambda = \sum_{y\in Y_{\mathrm{br}},\lambda_y=\lambda}\left(\textstyle{\left\lfloor \frac{2m_y+3}{4}\right\rfloor} - \textstyle{\left\lfloor \frac{m_y+3}{4}\right\rfloor}+\textstyle{\frac{M_y-m_y}{2}} \right),$$
$b= \left(\sum_{y\in Y_{\mathrm{br}}} \left\lfloor \frac{m_y+3}{4}\right\rfloor\right)-1$, $c= \sum_{y\in Y_{\mathrm{br}}} \left(\left\lfloor \frac{3m_y+3}{4}\right\rfloor - \left\lfloor \frac{2m_y+3}{4}\right\rfloor\right)$, and $d=g(Y)$. Moreover, if $y\in Y_{\mathrm{br}}$ satisfies condition $(\mathrm{b})$ above then we must have $m_y\equiv 1\mod 4$ and $\delta_y=\left\lfloor \frac{m_y}{4}\right\rfloor$.
\end{theorem}

\begin{proof}
Suppose $\lambda=a/b \not\in \Lambda_{\mathrm{br}}$. By Proposition \ref{prop:dims}, we obtain
\begin{equation}
\label{eq:dumb1}
r_2 - r_1= \sum_{y\in Y_{\mathrm{br}}} \left( \left\lfloor \frac{2m_y+3}{4}\right\rfloor - \left\lfloor \frac{m_y+3}{4}\right\rfloor +\frac{M_y-m_y}{2}\right).
\end{equation}
By Theorem \ref{thm:mostgeneral} and the assumptions (a)-(c) in Theorem \ref{thm:onlyB1orB2B3small}, it follows that the right side of (\ref{eq:dumb1}) equals $\varepsilon_1$. Using the equations for $r_1, r_2$ from the proof of Theorem \ref{thm:mostgeneral}, we see that
\begin{equation}
\label{eq:dumb2}
\sum_{h=1}^{\varepsilon_1}  \ell_h + \sum_{i=1}^{\varepsilon_2} (n_i-1) +\sum_{j=1}^{\varepsilon_3} n'_j = r_2-r_1.
\end{equation}
Hence we obtain that
$$\sum_{h=1}^{\varepsilon_1}  (\ell_h - 1) + \sum_{i=1}^{\varepsilon_2} (n_i-1) +\sum_{j=1}^{\varepsilon_3} n'_j = 0 .$$
This implies $\ell_h=1=n_i$ for all $1\le h \le \varepsilon_1,1\le i\le \varepsilon_2$, and $\varepsilon_3=0$. In case (b), equation (\ref{eq:duh1new}) and the inequality $\delta_y\le \left\lfloor \frac{m_y}{4}\right\rfloor$ shows that we must have $m_y\equiv 1$ mod $4$ and $\delta_y=\left\lfloor \frac{m_y}{4}\right\rfloor$. This completes the proof of Theorem \ref{thm:onlyB1orB2B3small}.
\end{proof}

\begin{example}
\label{ex:Nickthesis}
Suppose the field $k(X)$ is generated over $k(Y)$ by an element $T\in k(X)$ such that $T^4-T=\omega$ for an element $\omega\in k(Y)$ satisfying the following property: For each $y\in Y$, there exists $t(y)\in k(Y)$ such that $\mathrm{ord}_y(\omega-(t(y)^4-t(y)))$ is either (i) non-negative or (ii) negative and odd. Moreover, assume that (ii) occurs for at least one point $y\in Y$. 

This case was studied in \cite[Section 3]{CamachoThesis2021} and the precise decomposition of $\HH^0(X,\Omega_{X/k})$ into a direct sum of indecomposable $kG$-modules was determined. In particular, this case generalizes the main result of \cite{RzedowskiCVillaSMadan1996} when $p=2=n$ from $Y=\mathbb{P}^1_k$ to an arbitrary smooth projective curve $Y$.

We claim that this case implies that $Y_{\mathrm{br}} = B_1$ and $\Lambda_{\mathrm{br}}=\{\alpha\}$ where $\langle \alpha\rangle = \mathbb{F}_4^\times$, so that it is a subcase of Theorem \ref{thm:onlyB1orB2B3small}. To see this, let $u=T^2-T$ and $v=\alpha T^2-\alpha^2 T$. Then $u^2-u=\omega$ and $v^2-v=\alpha^2\omega$. Therefore, if $y\in Y$ is such that $\mathrm{ord}_y(\omega-(t(y)^4-t(y)))=-m_y$ is negative and odd, then $s_{p,y}:=t(y)^2-t(y)$, $s_{q,y}:=\alpha t(y)^2-\alpha^2 t(y)$ and $s_{p+q,y}:=\alpha^2 t(y)^2-\alpha t(y)$ satisfy
$$\mathrm{ord}_y(\omega-(s_{p,y}^2-s_{p,y}))=-m_y=\mathrm{ord}_y(\alpha^2\omega-(s_{q,y}^2-s_{p,y}))
=\mathrm{ord}_y(\alpha\omega-(s_{p+q,y}^2-s_{p+q,y})).$$
In other words, letting $p=\omega$, $q=\alpha^2\omega$, we see that the cover $\pi:X\to Y$ satisfies Assumption $\ref{ass:totallyramified}$ (see also \cite[Proposition 3.7.10]{Stichtenoth2009}). Moreover, if $y\in Y$ is a branch point and $x\in X$ lies above it, then, letting $u_y=u-s_{p,y}$ and $v_y=v-s_{q,y}$, we see that $w_y=v_y+\alpha u_y$ satisfies $\mathrm{ord}_x(w_y) = -m_y$. Therefore, we obtain $Y_{\mathrm{br}} = B_1$ and $\Lambda_{\mathrm{br}}=\{\alpha\}$. By Theorem \ref{thm:onlyB1orB2B3small}, it follows that
$$\HH^0(X,\Omega_{X/k})\cong N_{2,\alpha}^{\oplus a} \;\oplus\; M_{3,1}^{\oplus b} \;\oplus\; k^{\oplus c} \;\oplus\; kG^{\oplus d}$$
where $a = \sum_{y\in Y_{\mathrm{br}}} \left(\left\lfloor \frac{2m_y+3}{4}\right\rfloor - \left\lfloor \frac{m_y+3}{4}\right\rfloor \right)$, $b= \left(\sum_{y\in Y_{\mathrm{br}}} \left\lfloor \frac{m_y+3}{4}\right\rfloor\right)-1$, $c= \sum_{y\in Y_{\mathrm{br}}} \left(\left\lfloor \frac{3m_y+3}{4}\right\rfloor - \left\lfloor \frac{2m_y+3}{4}\right\rfloor\right)$, and $d=g(Y)$. 
\end{example}

The next two examples (Examples \ref{ex:oops1} and \ref{ex:oops2}) show that the list of isomorphism classes of indecomposable $kG$-modules that actually occur, for various $X$, as direct summands of $\HH^0(X,\Omega_{X/k})$ in the situation of Theorem \ref{thm:onlyB1orB2B3small} is infinite and given as follows:
\begin{equation}
\label{eq:list1}
\left\{N_{2,\lambda}\;:\;\lambda\in k \cup\{\infty\}\right\} \cup \left\{M_{3,1},k,kG\right\}.
\end{equation}
Moreover, we use these examples to show that the statement of \cite[Theorem 6.4]{MarquesWard2018} is incorrect.

Notice that $kG$ will occur with multiplicity at least $1$ as soon as the genus of $Y$ is bigger than 0. For this reason, we restrict ourselves to examples when $Y=\mathbb{P}^1_k$ to illustrate that the other indecomposable $kG$-modules in the set (\ref{eq:list1}) all do occur as direct summands of $\HH^0(X,\Omega_{X/k})$ for various $X$.

\begin{example}
\label{ex:oops1}
Let $Y=\mathbb{P}^1_k$ with function field $k(Y)=k(t)$, and fix $\lambda_0\in k-\{0,1\}$. We first provide an example in which $\HH^0(X,\Omega_{X/k})$ contains simultaneously as direct summands the 4 non-isomorphic 2-dimensional $kG$-modules $N_{2,\infty}$, $N_{2,0}$, $N_{2,1}$ and $N_{2,\lambda_0}$. Moreover, we will show that $k(X)$ is a global standard function field, in the sense of \cite[Definitions 2.1 and 3.1]{MarquesWard2018}. Hence, this example contradicts \cite[Theorem 6.4]{MarquesWard2018} which states that at most $3$ non-isomorphic indecomposable $kG$-modules that have $k$-dimensions $1$, $2$ and $3$, respectively, should occur as direct summands of $\HH^0(X,\Omega_{X/k})$.

Let $\alpha \in k-\{\lambda_0^{-1}\}$ satisfy the equation $\alpha^2+\alpha+1=0$, i.e. $\mathbb{F}_4=\{0,1,\alpha,\alpha^2\}$. Define $\beta:=(1+\alpha\lambda_0)^{-1}\in k-\{0,1,\alpha\}$. Let $p$ and $q$ in (\ref{eq:field}) be given by
\begin{eqnarray*}
p&=&\frac{1}{t(t-1)^3(t-\alpha)^3(t-\beta)^3},\\
q&=&\frac{\alpha}{t^3(t-1)(t-\alpha)^3(t-\beta)^3} ,
\end{eqnarray*}
which implies
$$p+q
=\frac{\alpha^2}{t^3(t-1)^3(t-\alpha)(t-\beta)^3}.$$
For $\mu\in k\cup\{\infty\}$, let $y_\mu$ be the corresponding closed point in $\mathbb{P}^1_k$ with uniformizer $\pi_{y_\mu}=t-\mu$ for $\mu\in k$ and $\pi_{y_\infty} = t^{-1}$. Then
$$\mathrm{ord}_{y_\mu}(p), \mathrm{ord}_{y_\mu}(q), \mathrm{ord}_{y_\mu}(p+q)\in \{-1,-3\}\quad
\mbox{if $\mu\in 
\{0,1,\alpha,\beta\}
$,}$$
and $\mathrm{ord}_{y_\mu}(p), \mathrm{ord}_{y_\mu}(q), \mathrm{ord}_{y_\mu}(p+q)$ are non-negative for all 
$\mu \in k\cup\{\infty\}-\{0,1,\alpha,\beta\}$.
In particular, $Y_{\mathrm{br}} = \{y_0,y_1,y_\alpha,y_{\beta}\}$, and $m_{y_\mu}=1,M_{y_\mu}=3$ for $\mu\in\{0,1,\alpha\}$, and $m_{y_{\beta}}=M_{y_{\beta}}=3$. Define $w:=v+  \alpha^2\,\frac{t-1}{t} \,u$. Then 
$$w^2 - w = 
\left(\alpha\,\frac{(t-1)^2}{t^2} - \alpha^2\,\frac{t-1}{t} \right)\, u
= \frac{(t-1)(t-\alpha)}{t^2}\,u.$$
Hence, if $x_\mu\in X$ lies above $y_\mu$, for $\mu\in k\cup\{\infty\}$, we obtain 
$$\mathrm{ord}_{x_\mu}(w) = \left\{\begin{array}{cl} -5 & \mbox{if } \mu=0,\\ -3 & \mbox{if } \mu=\beta,\\ -1 & \mbox{if } \mu\in\{1,\alpha\},\end{array}\right. $$
and $\mathrm{ord}_{x_\mu}(w)$ is non-negative for all $\mu\in k\cup\{\infty\}-\{0,1,\alpha,\beta\}$. In other words, $k(X)/k(t)$ is a global standard function field, in the sense of \cite[Definitions 2.1 and 3.1]{MarquesWard2018}.  Since $m_{y_\beta}=3$ and $\alpha^2\,\frac{\beta-1}{\beta}=\lambda_0$, it follows from Notations \ref{not:twocases} and \ref{not:branchpoints} that $\delta_{y_\beta}=0$ and $\lambda_{y_\beta}=\lambda_0$. Therefore,
$$y_0\in B_{3,\infty}, \quad y_1\in B_{3,0},\quad y_{\alpha}\in B_{3,1},\quad\mbox{and}\quad
y_{\beta}\in B_{1,\lambda_0}.$$
Since $m_{y_\mu}=1$ for $\mu\in\{0,1,\alpha\}$, we have that $y_0$, $y_1$, $y_\alpha$ satisfy condition (c) of Theorem \ref{thm:onlyB1orB2B3small}, whereas $y_{\beta}$ satisfies condition (a) of this theorem. Hence this theorem shows that
$$\HH^0(X,\Omega_{X/k})\cong  N_{2,\infty} \;\oplus\; N_{2,0} \;\oplus\; N_{2,1} \;\oplus\; N_{2,\lambda_0} 
\;\oplus\; M_{3,1}^{\oplus 3} \;\oplus\; k$$
showing that there are $6$ non-isomorphic indecomposable $kG$-modules that occur as direct summands of $\HH^0(X,\Omega_{X/k})$, including $4$ non-isomorphic indecomposable $kG$-modules of $k$-dimension $2$. 
\end{example}

\begin{example}
\label{ex:oops2}
Let $Y=\mathbb{P}^1_k$ with function field $k(Y)=k(t)$. Given a positive integer $n$, we next provide an example in which there are $n+2$ non-isomorphic indecomposable $kG$-modules that occur as direct summands of $\HH^0(X,\Omega_{X/k})$. Moreover, we will show that $k(X)$ is a global standard function field, in the sense of \cite[Definitions 2.1 and 3.1]{MarquesWard2018}. As in Example \ref{ex:oops1}, we obtain a contradiction to \cite[Theorem 6.4]{MarquesWard2018} provided $n\ge 2$.

Let 
$\lambda_1,\ldots,\lambda_n \in k-\{0,1\}$ be $n$ distinct elements,
and let $p$ and $q$ in (\ref{eq:field}) be given by
\begin{eqnarray*}
p&=&\frac{1}{(t-\lambda_1)^5\cdots (t-\lambda_n)^5},\\
q&=&\frac{t^2}{(t-\lambda_1)^5\cdots (t-\lambda_n)^5} .
\end{eqnarray*}
For $1\le i\le n$, let $y_i$ be the point in $\mathbb{A}^1_k\subset \mathbb{P}^1_k$ corresponding to $\lambda_i$ with uniformizer $\pi_{y_i}=t-\lambda_i$. Then
$$\mathrm{ord}_{y_i}(p) = -5 = \mathrm{ord}_{y_i}(q) = \mathrm{ord}_{y_i}(p+q).$$
This means that $Y_{\mathrm{br}}=\{y_1,\ldots,y_n\}$, and $m_{y_i}=5=M_{y_i}$ for all $1\le i\le n$. Define
$w:=v+  t \,u$. Then 
$$w^2 - w = 
t ( t -1) u.$$
Since none of the $y_i$ correspond to either $0$ or $1$ in $k$, we obtain, for $1\le i \le n$, that $\mathrm{ord}_{x_i}(w) = -5$ when $x_i\in X$ lies above $y_i$. In other words, $k(X)/k(t)$ is a global standard function field, in the sense of \cite[Definitions 2.1 and 3.1]{MarquesWard2018}. 

Let $i\in\{1,\ldots,n\}$.
Using Remark \ref{rem:twocases}, it follows that $\beta_{y_i} = \lambda_i + \pi_{y_i}$. Therefore, in Notation \ref{not:twocases}, we have $\lambda_{y_i}=\lambda_i$ and $\delta_{y_i}=1$. 
Since $\left\lfloor \frac{2m_{y_i}+3}{4}\right\rfloor - \left\lfloor \frac{m_{y_i}+3}{4}\right\rfloor = 1$, all $y_i$ satisfy condition (b) of Theorem \ref{thm:onlyB1orB2B3small}. Hence this theorem shows that
$$\HH^0(X,\Omega_{X/k})\cong  N_{2,\lambda_1}\;\oplus\; \cdots \;\oplus\; N_{2,\lambda_n} 
\;\oplus\; M_{3,1}^{\oplus (2n-1)} \;\oplus\; k^{\oplus n}.$$
In particular, there are $n+2$ non-isomorphic indecomposable $kG$-modules that occur as direct summands of $\HH^0(X,\Omega_{X/k})$, including $n$ non-isomorphic indecomposable $kG$-modules of $k$-dimension $2$.
\end{example}

Theorem \ref{thm:awkward1} follows from Remarks \ref{rem:twocases} and \ref{rem:awkward1follows}, Theorem \ref{thm:onlyB1orB2B3small}, and Examples \ref{ex:oops1} and \ref{ex:oops2}. Note that $y\in Y_{\mathrm{br}}$ satisfies condition (ii) in Theorem \ref{thm:awkward1} if and only if  $y$ satisfies either condition (a) or condition (c) in Theorem \ref{thm:onlyB1orB2B3small} (see Remark \ref{rem:awkward1follows}).

\begin{remark}
\label{rem:wrongMarquesWard}
In the proof of \cite[Lemma 6.3]{MarquesWard2018}, the first of the displayed equations for $\sigma^h(w_{\mu,\nu})$ is not correct in general. As a consequence, the $k$-vector space $\Delta_{\mu,\nu}$ defined in \cite[Theorem 6.4]{MarquesWard2018} is not always a $kG$-module. 
Examples \ref{ex:oops1} and \ref{ex:oops2} provide specific examples where this happens.
\end{remark}

\begin{remark}
\label{rem:candomore}
In all cases that are not covered by conditions (a)-(c) in Theorem \ref{thm:onlyB1orB2B3small}, it is tougher to find the actual values of $\ell_h, n_i, n'_j$ for $1\le h \le \varepsilon_1,1\le i\le \varepsilon_2,1\le j\le \varepsilon_3$, in the notation of Definition \ref{def:decompose}. However, here is a situation in which we can say more. Namely, using the notation of Theorem \ref{thm:mostgeneral}, suppose
\begin{equation}
\label{eq:sodumb!}
\sum_{h=1}^{\varepsilon_1} \ell_h \ge \sum_{y\in Y_{\mathrm{br}}} \left( \textstyle{\left\lfloor \frac{2m_y+3}{4}\right\rfloor} - \textstyle{\left\lfloor \frac{m_y+3}{4}\right\rfloor} + \textstyle{\frac{M_y-m_y}{2}}\right).
\end{equation}
Using the same arguments as for equations (\ref{eq:dumb1}) and (\ref{eq:dumb2}) in the proof of Theorem \ref{thm:onlyB1orB2B3small}, we then obtain that
$$ \sum_{i=1}^{\varepsilon_2} (n_i-1) +\sum_{j=1}^{\varepsilon_3} n'_j \le 0$$
which implies that we must have $n_i=1$ for all $1\le i\le \varepsilon_2$, and $\varepsilon_3=0$. In particular, we must then have 
$$\varepsilon_4 = \displaystyle \sum_{y\in Y_{\mathrm{br}}} \left(\left\lfloor \frac{3m_y+3}{4}\right\rfloor - \left\lfloor \frac{2m_y+3}{4}\right\rfloor\right).$$
In \S \ref{s:KleinfourP1}, we will show that if $Y=\mathbb{P}^1_k$ then the inequality in (\ref{eq:sodumb!}) is satisfied and actually an equality, and we will give the precise description of the $kG$-module structure of $\HH^0(X,\Omega_{X/k})$ in this case (see Theorem \ref{thm:P1main}). 
\end{remark}


\section{Klein four covers of the projective line}
\label{s:KleinfourP1}

Suppose $Y=\mathbb{P}^1_k$ and $\pi:X\to \mathbb{P}^1_k$ is a $G$-cover satisfying Assumptions \ref{ass:general} and \ref{ass:totallyramified}. In this section, we determine the precise $kG$-module structure of $\HH^0(X,\Omega_{X/k})$. 

We use the following notation. 

Write the function field $K=k(Y)=k(t)$ for a variable $t$. For $\mu\in k\cup\{\infty\}$, let $y_\mu$ be the corresponding closed point in $\mathbb{P}^1_k$ with uniformizer
\begin{equation}
\label{eq:uniformizer}
\pi_{y_\mu}=\left\{\begin{array}{cl}
t^{-1} & \mbox{ if $\mu=\infty$},\\
t-\mu & \mbox{ if $\mu\ne \infty$}.
\end{array}\right.
\end{equation}
For simplicity, we will write $\infty$ instead of $y_\infty$ in what follows.
If $f$ is any rational function in $K=k(t)$, then we can use its partial fraction decomposition to bring $f$ into standard form for all the primes of $k(t)$ (as described by Hasse in \cite{HasseCrelle1935}). More precisely, since $k$ is algebraically closed of characteristic 2, there exists a rational function $s_f\in K=k(t)$ such that
\begin{equation}
\label{eq:Hassestandard}
f- (s_f^2-s_f)=\frac{N_f(t)}{\prod_i D_{f,i}(t)^{m_{f,i}}}
\end{equation}
where $D_{f,i}(t)=t-a_{f,i}$ for pairwise distinct $a_{f,i}\in  k$, $m_{f,i}$ are positive odd integers, and the polynomial $N_f(t)$ is relatively prime to the denominator.
Because of Assumption \ref{ass:totallyramified} and Remark \ref{rem:twocases}, we can change the variable $t$ to be able to make the following assumptions throughout this section.

\begin{assume}
\label{ass:overP1}
Let $Y=\mathbb{P}^1_k$ with function field $K=k(Y)=k(t)=k(t^{-1})$. Let $p,q\in K-k$ be as in (\ref{eq:field}), and define
\begin{equation}
\label{eq:r!}
r:=p+q.
\end{equation} 
There exist $n\ge 0$ distinct elements $\lambda_1,\ldots,\lambda_n\in k-\{0\}$ with corresponding closed points $y_1,\ldots, y_n\in\mathbb{A}^1_k-\{0\}\subset\mathbb{P}^1_k$ such that the following is true: 

For $f\in\{p,q,r\}$, there exists a rational  function $s_f$ in $K$ such that
\begin{equation}
\label{eq:frational}
f - (s_f^2-s_f) = \frac{f_1(t^{-1})}{f_2(t^{-1})}
\end{equation}
where $f_1(t^{-1}),f_2(t^{-1})$ are relatively prime polynomials  in $k[t^{-1}]$, and there exist positive odd integers $m_{f,\infty},m_{f,y_1},\ldots,m_{f,y_n}$ such that
\begin{equation}
\label{eq:fdenominator}
f_2(t^{-1}) = (t^{-1})^{m_{f,\infty}}(t^{-1}-\lambda_1^{-1})^{m_{f,y_1}} \cdots (t^{-1}-\lambda_n^{-1})^{m_{f,y_n}}.
\end{equation}
\end{assume}

\begin{remark}
\label{rem:overP1}
Under Assumptions \ref{ass:general} and \ref{ass:overP1}, the finite set of branch points is given as 
$$Y_{\mathrm{br}}=\{\infty,y_1,\ldots,y_n\} .$$
Moreover, if $y\in Y_{\mathrm{br}}$ then $\{m_y,M_y\} = \{m_{p,y},m_{q,y},m_{r,y}\}$, and if $m_y < M_y$ then precisely two of $m_{p,y}$, $m_{q,y}$, $m_{r,y}$ are equal to $M_y$. 
For $f\in\{p,q,r\}$, suppose $d_f$ is the degree of $f_1(t^{-1})$ in $t^{-1}$, and write
\begin{equation}
\label{eq:numerator}
f_1(t^{-1}) = c_{f,d_f}+ \cdots + c_{f,1} t^{-d_f+1}  + c_{f,0} t^{-d_f} 
\end{equation}
in $k[t^{-1}]$ where $c_{f,0}\ne 0$. Since $t^{-1}$ is not a factor of $f_1(t^{-1})$, we also have that $c_{f,d_f}\ne 0$. By multiplying both $f_2$ in (\ref{eq:fdenominator}) and $f_1$ in (\ref{eq:numerator}) by the same appropriate non-zero scalar and  power of $t$, we can rewrite (\ref{eq:frational}) as
$$f - (s_f^2-s_f)= t^{m_{f,\infty}+m_{f,y_1}+\cdots + m_{f,y_n}-d_f} \left(\lambda_1^{m_{f,y_1}} \cdots \lambda_n^{m_{f,y_n}} \right) \frac{c_{f,d_f}t^{d_f} + \cdots + c_{f,1} t + c_{f,0}}{ (t-\lambda_1)^{m_{f,y_1}} \cdots (t-\lambda_n)^{m_{f,y_n}}} .$$
Since $0\not\in Y_{\mathrm{br}}$, it follows that 
\begin{equation}
\label{eq:rationaldegree}
d_f \le m_{f,\infty}+\sum_{i=1}^n m_{f,y_i}\qquad \mbox{for $f\in\{p,q,r\}$.}
\end{equation}
\end{remark}

\begin{remark}
\label{rem:needoverP1}
Let $u,v,p,q,r,s_p,s_q,s_r$ be as in Assumption \ref{ass:overP1}, where $r=p+q$. Define
\begin{equation}
\label{eq:yikes}
\left\{ \quad\begin{array}{lcllcl}
\widetilde{u}&:=&u-s_p, & \widetilde{p}&:=&p-(s_p^2-s_p),\\
\widetilde{v}&:=&v-s_q, & \widetilde{q}&:=&q-(s_q^2-s_q),\\
\widetilde{u+v}&:=&u+v-s_r, & \widetilde{r}&:=&r-(s_r^2-s_r) .
\end{array}\right.
\end{equation}
Fix $y\in Y_{\mathrm{br}}$. By Remarks \ref{rem:twocases} and \ref{rem:overP1}, we obtain the following values for $m_y$ and $M_y$:
\begin{itemize}
\item
If $m_{p,y}= m_{q,y}=m_{r,y}$ or $m_{p,y}< m_{q,y}=m_{r,y}$, then $m_y=m_{p,y}$ and $M_y=m_{q,y}=m_{r,y}$, and
$$u_y=\widetilde{u}, \; p_y=\widetilde{p}, \qquad v_y=\widetilde{v}, \; q_y=\widetilde{q}.$$
\item
If $m_{q,y}< m_{p,y}=m_{r,y}$, then $m_y=m_{q,y}$ and $M_y=m_{p,y}=m_{r,y}$, and
$$u_y=\widetilde{v}, \; p_y=\widetilde{q}, \qquad v_y=\widetilde{u}, \; q_y=\widetilde{p}.$$
\item
If $m_{r,y}< m_{p,y}=m_{q,y}$, then $m_y=m_{r,y}$ and $M_y=m_{p,y}=m_{q,y}$, and
$$u_y=\widetilde{u+v}, \; p_y=\widetilde{r}, \qquad v_y=\widetilde{v}, \; q_y=\widetilde{q}.$$
\end{itemize}
\end{remark}

\begin{remark}
\label{rem:holodiffoverP1}
For each $y\in Y=\mathbb{P}^1_k$, let $\pi_y$ be the corresponding uniformizer, as defined in (\ref{eq:uniformizer}). 
Then the differentials of the uniformizers are given as
\begin{equation}
\label{eq:diffuni}
d\pi_y = \left\{\begin{array}{cl} t^{-2} dt & \mbox{if } y=\infty,\\
dt & \mbox{if } y\neq \infty.
\end{array}\right.
\end{equation}
The stalk of $\Omega_{Y/k}$ at $y$ is given as
$$(\Omega_{Y/k})_y=\mathcal{O}_{Y,y}\,d\pi_y = \{ h\,dt\; : \; h \in k(Y), \; \mathrm{ord}_y(h) \ge -k_y\}$$
where 
\begin{equation}
\label{eq:ky}
k_y = \left\{\begin{array}{rl}
-2 & \mbox{ if $y=\infty$},\\
0 & \mbox{ if $y\ne \infty$}.
\end{array}\right.
\end{equation}
It follows from the proof of \cite[Proposition IV.2.3]{Hartshorne1977} that there is a natural isomorphism of $\mathcal{O}_X$-$G$-modules
$$\Omega_{X/k} = \pi^* \Omega_{Y/k}\otimes_{\mathcal{O}_X} \mathcal{D}_{X/Y}^{-1}.$$
In particular, we obtain that if $x_y\in X$ lies above $y\in Y$ then the stalk at $x_y$ equals
\begin{eqnarray*}
(\Omega_{X/k})_{x_y} &=& (\pi^* \Omega_{Y/k})_{x_y}\otimes_{\mathcal{O}_{X,x_y}} (\mathcal{D}_{X/Y}^{-1})_{x_y}\\
&=& \{ f\,dt\; : \; f \in k(X), \; \mathrm{ord}_{x_y}(f) \ge -e_{x_y/y}k_y-d_{x_y/y}\}
\end{eqnarray*}
where $e_{x_y/y}$ is the ramification index and $d_{x_y/y}$ is the different exponent. We have $e_{x_y/y}=1$ and $d_{x_y/y}=0$ if $y\not\in Y_{\mathrm{br}}$. Moreover, we have
\begin{equation}
\label{eq:differentP1}
e_{x_y/y} = 4\quad\mbox{ and }\quad d_{x_y/y} = 3(m_y+1)+2(M_y-m_y) \qquad \mbox{for $y\in Y_{\mathrm{br}}$ and $x_y\in X$ above $y$}
\end{equation}
(see Remark \ref{rem:twocases}).
We obtain the following natural identifications as $kG$-modules:
\begin{eqnarray*}
\HH^0(X,\Omega_{X/k}) &=& \bigcap_{x\in X} (\Omega_{X/k})_x\\
&=& \{ f\,dt\; : \; f \in k(X), \mathrm{ord}_x(f) \ge -e_{x/\pi(x)}k_{\pi(x)}-d_{x/\pi(x)} \mbox{ for all } x\in X\}.
\end{eqnarray*}
By the Riemann-Hurwitz formula, we have
\begin{equation}
\label{eq:genus}
g(X) = 1+4(g(Y)-1) + \frac{1}{2} \sum_{y\in Y_{\mathrm{br}}} d_{x_y/y}= -3 + \frac{1}{2} \sum_{y\in Y_{\mathrm{br}}} d_{x_y/y}
\end{equation}
\end{remark}

\begin{definition}
\label{def:basisoverP1}
Fix $y\in Y_{\mathrm{br}}$. Let $u_y,v_y$ be as in Remark \ref{rem:needoverP1}, let $\alpha_y,\beta_y,\lambda_y, \delta_y$ be as in Remark \ref{rem:twocases}, let $w_y(j)$, for $j\ge 0$, be as in Lemma \ref{lem:twocases}, and let $\pi_y,k_y$ be as in Remark \ref{rem:holodiffoverP1}.
\begin{itemize}
\item[(a)]
Define $\mu_{y,1}:=\left\lfloor \frac{m_y+3}{4} \right\rfloor$, $\mu_{y,2}:=\left\lfloor \frac{2m_y+3}{4} \right\rfloor$, $\mu_{y,3}:=\left\lfloor \frac{3m_y+3}{4} \right\rfloor$, and $\nu_y:=\frac{M_y-m_y}{2}$. Moreover define
$$s(y):=\left\{ \begin{array}{cl} 0&\mbox{ if $y=\infty$},\\
1&\mbox{ if $y\ne \infty$}.
\end{array}\right.$$
\item[(b)] Define 
$$\begin{array}{rclccl}
f_{y,1,i_1(y)} &:=& \pi_y^{-i_1(y)}, & s(y)&\le i_1(y) \le& \mu_{y,3}  + \nu_y + k_y ,\\[1ex]
f_{y,2,i_2(y)} &:=& \pi_y^{-i_2(y)}\, u_y, & s(y)&\le i_2(y) \le& \mu_{y,1} + \nu_y + k_y ,\\[1ex]
f_{y,3,i_3(y)} &:=& \pi_y^{-i_3(y)}\, v_y, & s(y)&\le i_3(y) \le& \mu_{y,1} + k_y  ,\\[1ex]
f_{y,3,i_3(y)} &:=& \pi_y^{-i_3(y)}\, w_{y}\left(i_3(y) - \mu_{y,1} - k_y - 1\right), \quad& \mu_{y,1} + k_y +1  &\le i_3(y) \le& \mu_{y,2}  + k_y ,
\end{array}$$
and define 
\begin{eqnarray*}
\mathcal{B}_{y,1} &:=& \{f_{y,1,i_1(y)}\; : \; s(y)\le i_1(y) \le \mu_{y,3}  + \nu_y + k_y \} ,\\
\mathcal{B}_{y,2} &:=& \{f_{y,2,i_2(y)}\; : \; s(y)\le i_2(y) \le \mu_{y,1} + \nu_y + k_y \},\\
\mathcal{B}_{y,3} &:=& \{f_{y,3,i_3(y)}\; : \; s(y)\le i_3(y) \le \mu_{y,2}  + k_y \}.
\end{eqnarray*}
\item[(c)]
For each $y\in Y_{\mathrm{br}}$, define $\mathcal{B}_y=\mathcal{B}_{y,1} \cup \mathcal{B}_{y,2} \cup \mathcal{B}_{y,3}$. Define
$$\mathcal{B} := \bigcup_{y\in Y_{\mathrm{br}}} \mathcal{B}_y.$$
\end{itemize}
\end{definition}

\begin{lemma}
\label{lem:basisoverP1}
Under Assumptions $\ref{ass:general}$ and $\ref{ass:overP1}$, a $k$-basis for $\HH^0(X,\Omega_{X/k})$ is given by $\{f\,dt\; :\; f\in\mathcal{B}\}$ where $\mathcal{B}$ is as in Definition $\ref{def:basisoverP1}$.
\end{lemma}

\begin{proof}
Fix $y\in Y_{\mathrm{br}}$ and $x_y\in X$ above $y$. By (\ref{eq:differentP1}), we have the following equalities of positive integers:
\begin{equation}
\label{eq:duh0}
\begin{array}{cclccl}
\left\lfloor \frac{d_{x_y/y}}{4} \right\rfloor &=& \mu_{y,3} + \nu_y,\quad&
\left\lfloor \frac{d_{x_y/y}-2m_y}{4} \right\rfloor &=& \mu_{y,1} + \nu_y,\\[1ex]
\left\lfloor \frac{d_{x_y/y}-2M_y}{4} \right\rfloor &=& \mu_{y,1},&
\left\lfloor \frac{d_{x_y/y}-m_y-2(M_y-m_y)}{4} \right\rfloor &=& \mu_{y,2}.
\end{array}
\end{equation}
By (\ref{eq:duh1new}) we have
\begin{equation}
\label{eq:duh1}
\mu_{y,2}-\mu_{y,1} 
= \left\{\begin{array}{ll} \textstyle{\left\lfloor \frac{m_y}{4}\right\rfloor}  & \mbox{if } m_y\equiv 1 \mod 4, \mbox{ and}\\
\textstyle{\left\lfloor \frac{m_y}{4}\right\rfloor +1} & \mbox{if } m_y\equiv 3 \mod 4.
\end{array}\right.
\end{equation}

Suppose $\mathrm{b}_y\in\mathcal{B}_y$. Using (\ref{eq:duh0}), (\ref{eq:duh1}) and Lemma \ref{lem:twocases}, we see that $\mathrm{ord}_{x_y}(\mathrm{b}_y)\ge -e_{x_y/y}k_y-d_{x_y/y}$.
Suppose now that $y'\in Y_{\mathrm{br}}-\{y\}$. Since $u_y,v_y\in\{\widetilde{u},\widetilde{v},\widetilde{u+v}\}$, we obtain that both $\mathrm{ord}_{x_{y'}}(u_y)$ and $\mathrm{ord}_{x_{y'}}(v_y)$ are greater than or equal to $-2M_{y'}$. Moreover, for $a\in \mathbb{Z}$, it follows from (\ref{eq:uniformizer}) and Assumption \ref{ass:overP1} that $\mathrm{ord}_{x_{\infty}}(\pi_y^a)=-4a$ and $\mathrm{ord}_{x_{y'}}(\pi_y^a) = 0 = \mathrm{ord}_{x_{y'}}(\pi_\infty^a)$ when $y\ne \infty$ and $y'\not\in\{y,\infty\}$.
Using this, we obtain that if $y'\in Y_{\mathrm{br}}-\{y\}$ then $\mathrm{ord}_{x_{y'}}(\mathrm{b}_y)\ge -e_{x_{y'}/y'}k_{y'}-d_{x_{y'}/y'}$. 

Finally, suppose that $x\in X$ with $\pi(x)\not\in Y_{\mathrm{br}}$. Since $u_y,v_y \in \{\widetilde{u},\widetilde{v},\widetilde{u+v}\}$, we obtain from (\ref{eq:rationaldegree}) that both $\mathrm{ord}_x(u_y)$ and $\mathrm{ord}_x(v_y)$ are greater than or equal to $0$. Using this, we see that $\mathrm{ord}_x(\mathrm{b}_y)\ge 0$.

Therefore, $f\,dt\in \HH^0(X,\Omega_{X/k})$ for all $f\in\mathcal{B}$. Since $k(X)=k(Y)(u,v)=k(t)(\widetilde{u},\widetilde{v})$ has a $k(t)$-basis given by $\{1,\widetilde{u},\widetilde{v},\widetilde{u} \widetilde{v}\}$ and since 
$$\widetilde{u+v}=\widetilde{u}+\widetilde{v} + \left(s_p+s_q-s_r\right) $$
by (\ref{eq:yikes}), we obtain that the elements in $\mathcal{B}$ are $k$-linearly independent.  

Using that $m_y$ is a positive odd integer, in addition to (\ref{eq:differentP1}) and (\ref{eq:duh0}), we see that
$$\mu_{y,3}+\nu_y + \mu_{y,1}+\nu_y + \mu_{y,2} = \textstyle \frac{1}{2}\,d_{x_y/y}.$$
Therefore, we obtain
$$\#\mathcal{B}_y = \left\{ \begin{array}{cl}
\frac{1}{2}\,d_{x_y/y} - 3&\mbox{if $y=\infty$},\\
\frac{1}{2}\,d_{x_y/y}&\mbox{if $y\ne\infty$},
\end{array}\right. $$
which implies by (\ref{eq:genus}) that $\#\mathcal{B}=g(X)$. 
This completes the proof of Lemma \ref{lem:basisoverP1}.
\end{proof}

The next theorem gives the precise decomposition of $\HH^0(X,\Omega_{X/k})$ into a direct sum of indecomposable $kG$-modules, where we use Notation \ref{not:indecomposables}  for the indecomposable $kG$-modules. Moreover, we define $N_{0,\lambda}$ to be the zero module for all $\lambda\in k\cup\{\infty\}$.

\begin{theorem}
\label{thm:P1main}
Under Assumptions $\ref{ass:totallyramified}$ and $\ref{ass:overP1}$, we have an isomorphism of $kG$-modules
$$\HH^0(X,\Omega_{X/k})\cong \bigoplus_{y\in Y_{\mathrm{br}}} \left(N_{2\ell_y,\lambda_y}^{\oplus a_{y,1}}\;\oplus\; N_{2(\ell_y-1),\lambda_y}^{\oplus a_{y,2}}\right) \;\oplus\; M_{3,1}^{\oplus b} \;\oplus\; k^{\oplus c}$$
where $b=\left(\sum_{y\in Y_{\mathrm{br}}} \left\lfloor \frac{m_y+3}{4}\right\rfloor\right)-1$, $c= \sum_{y\in Y_{\mathrm{br}}} \left( \left\lfloor \frac{3m_y+3}{4}\right\rfloor - \left\lfloor \frac{2m_y+3}{4}\right\rfloor \right)$, and $\ell_y,a_{y,1},a_{y,2}$ are given as follows for each $y\in Y_{\mathrm{br}}$:
\begin{itemize}
\item[(i)] If $\delta_y=0$, then $\ell_y=1$, $a_{y,1}=\left\lfloor \frac{2m_y+3}{4}\right\rfloor - \left\lfloor \frac{m_y+3}{4}\right\rfloor$, and $a_{y,2}=0$. 
\item[(ii)] If $\delta_y\ge 1$, then $\ell_y\ge 1$ and $1\le a_{y,1}\le \delta_y$ are uniquely determined by the equation
$$\textstyle\left\lfloor \frac{2m_y+3}{4}\right\rfloor - \left\lfloor \frac{m_y+3}{4}\right\rfloor = (\ell_y-1) \delta_y +a_{y,1},$$
and $a_{y,2}=\delta_y-a_{y,1}$. 
\item[(iii)] If $\delta_y=-1$, then $\ell_y\ge 1$ and $1\le a_{y,1}\le \frac{M_y-m_y}{2}$ are uniquely determined by the equation
$$\textstyle \left\lfloor \frac{2m_y+3}{4}\right\rfloor - \left\lfloor \frac{m_y+3}{4}\right\rfloor + \frac{M_y-m_y}{2} = (\ell_y-1) \frac{M_y-m_y}{2} +a_{y,1},$$
and $a_{y,2}=\frac{M_y-m_y}{2}-a_{y,1}$. 
\end{itemize}
\end{theorem}

\begin{proof} 
Fix $y\in Y_{\mathrm{br}}$, and use Remark \ref{rem:needoverP1} and Definition \ref{def:basisoverP1}. Define 
$$\begin{array}{lll}
\sigma_y:=\sigma, & \tau_y:=\tau &\mbox{ in the situation of Remark \ref{rem:twocases}(i) or (ii)(a),}\\
\sigma_y:=\tau, & \tau_y:=\sigma &\mbox{ in the situation of Remark \ref{rem:twocases}(ii)(b), and}\\
\sigma_y:=\sigma\circ\tau, & \tau_y:=\tau &\mbox{ in the situation of Remark \ref{rem:twocases}(ii)(c).}
\end{array}$$
Recall from Lemma \ref{lem:twocases} that for $0\le j \le \mu_{y,2}-\mu_{y,1}-1$, we have 
$$\beta_y(j) = \pi_y^{-\nu_y} \left(\sum_{i=0}^{j} b_{y,i}\pi_y^i\right)\quad\mbox{and}\quad
w_y(j) = v_y +\widetilde{\alpha}_y + \beta_y(j) \,u_y$$
where $\pi_y$ is as in (\ref{eq:uniformizer}). By (\ref{eq:Gaction}) and Remark \ref{rem:twocases}, we have the following $G$-actions on the elements of $\mathcal{B}_y = \mathcal{B}_{y,1} \cup \mathcal{B}_{y,2} \cup \mathcal{B}_{y,3}$: 
$$\begin{array}{c||c|c||l}
 & \sigma_y -1 & \tau_y -1\\ \hline\hline
f_{y,1,i_1(y)} & 0 & 0 & \mbox{for } s(y) \le i_1(y) \le \mu_{y,3} + \nu_y + k_y\\
f_{y,2,i_2(y)} & 0 & f_{y,1,i_2(y)} & \mbox{for } s(y) \le i_2(y) \le \mu_{y,1} + \nu_y + k_y \\
f_{y,3,i_3(y)} & f_{y,1,i_3(y)} & 0 & \mbox{for } s(y) \le i_3(y) \le \mu_{y,1}  + k_y\\
f_{y,3,i_3(y)} & f_{y,1,i_3(y)} & \displaystyle{\sum_{i=0}^{i_3(y)-\mu_{y,1} - k_y - 1} b_{y,i} \,f_{y,1,i_3(y)+\nu_y-i}} 
& \mbox{for } \mu_{y,1}  + k_y + 1 \le i_3(y) \le \mu_{y,2} + k_y.
\end{array}$$
Define the following, pairwise disjoint subsets of $\mathcal{B}_y$:
$$\begin{array}{rcll}
\mathcal{B}_{y,3,1,s_1(y)} &:=& \{f_{y,1,s_1(y)}, f_{y,2,s_1(y)},f_{y,3,s_1(y)}\} & \mbox{for } s(y)\le s_1(y)\le \mu_{y,1}+k_y,\\
\mathcal{B}_{y,1,s_2(y)} &:=& \{f_{y,1,s_2(y)}\} & \mbox{for } \mu_{y,2}  + \nu_y + k_y + 1 \le s_2(y) \le \mu_{y,3}  + \nu_y + k_y.
\end{array}$$
Considering the actions of $\sigma_y-1$ and $\tau_y-1$ on these sets, we see that the $k$-span of each $\mathcal{B}_{y,3,1,s_1(y)}$ gives a $kG$-module isomorphic to $M_{3,1}$, and the $k$-span of each $\mathcal{B}_{y,1,s_2(y)}$ gives the trivial $kG$-module $k$. In particular, the number of copies of $M_{3,1}$ we obtain this way equals
$$\mu_{y,1}+k_y-s(y)+1=\left\{\begin{array}{ll} \mu_{y,1}-1 & \mbox{if } y=\infty,\\
\mu_{y,1} & \mbox{if } y\neq \infty,
\end{array}\right.$$
whereas the number of copies of the trivial $kG$-module $k$ equals $\mu_{y,3}-\mu_{y,2}$.

We next analyze the actions of $\sigma_y-1$ and $\tau_y-1$ on the remaining elements of $\mathcal{B}_y$, which we collect in the following (ordered) set: 
\begin{eqnarray}
\label{eq:Blambday}
\mathcal{B}_{\lambda_y} &:=& \{f_{y,1,s_3(y)} \; : \; \mu_{y,1}+k_y+1 \le s_3(y) \le \mu_{y,2}  + \nu_y + k_y \}\\
\nonumber
&& \quad \cup \; \{f_{y,2,s_4(y)} \; : \; \mu_{y,1}+k_y+1 \le s_4(y) \le \mu_{y,1} + \nu_y + k_y\}\\
\nonumber
&& \quad \cup \; \{f_{y,3,s_5(y)} \; : \; \mu_{y,1}+k_y+1 \le s_5(y) \le \mu_{y,2} + k_y\}.
\end{eqnarray}
Define $n_y:= \mu_{y,2}-\mu_{y,1}+\nu_y$, so that $\mathcal{B}_{\lambda_y}$ has $2n_y$ elements. We have $n_y=0$ if and only if $\mu_{y,2}-\mu_{y,1}=0$ and $\nu_y=0$. Moreover, $\mu_{y,2}-\mu_{y,1}=0$ if and only if $m_y=1$.

Suppose first that $n_y=0$, i.e. $\mathcal{B}_{\lambda_y}=\emptyset$.  Then $\delta_y=0$ since $\lfloor \frac{m_y}{4}\rfloor=0$. Hence we are in part (i) of the statement of the theorem, which gives correctly $a_{y,1}=\mu_{y,2}-\mu_{y,1}=0$ and $a_{y,2}=0$ in this case.

Suppose from now on that $n_y\ge 1$. It follows that $\sigma_y-1$ and $\tau_y-1$ act on the ordered basis $\mathcal{B}_{\lambda_y}$ from (\ref{eq:Blambday}) as the following square matrices:
$$\sigma_y-1 \;\longleftrightarrow\; \left(\begin{array}{c|c} \mathbf{0}_{n_y}&C_{\lambda_y}\\ \hline \mathbf{0}_{n_y}&\mathbf{0}_{n_y}\end{array}\right)
\qquad\mbox{and}\qquad
\tau_y-1\;\longleftrightarrow\; \left(\begin{array}{c|c} \mathbf{0}_{n_y}&D_{\lambda_y}\\ \hline \mathbf{0}_{n_y}&\mathbf{0}_{n_y}\end{array}\right)  ,$$
where $\mathbf{0}_{n_y}$ is the $n_y\times n_y$ zero matrix and $C_{\lambda_y}$ and $D_{\lambda_y}$ are $n_y\times n_y$ matrices with $(i,j)$-entries (for $1\le i,j \le n_y$)
\begin{eqnarray*}
\left(C_{\lambda_y}\right)_{i,j}&=&\left\{\begin{array}{cl}
1 & \mbox{ if $j=\nu_y+i$ and  and $1\le i\le n_y-\nu_y$,}\\
0 & \mbox{ otherwise,}
\end{array}\right. \\
\left(D_{\lambda_y}\right)_{i,j}&=&\left\{\begin{array}{cl}
1 &\mbox{ if $i=j$ and $1\le i \le \nu_y$,}\\
b_{y,j-i} & \mbox{ if $\nu_y+1\le i\le j\le n_y$,}\\
0 & \mbox{ otherwise.}
\end{array}\right. 
\end{eqnarray*}

Suppose first that we are in parts (i) or (ii) of the statement of the theorem; in other words, we are in the situation of Remark \ref{rem:twocases}(i). Then $m_y=M_y$ and $\nu_y=0$, which means that $C_{\lambda_y}$ is the $n_y\times n_y$ identity matrix $\mathbf{1}_{n_y}$. By Notation \ref{not:twocases} and equation (\ref{eq:duh1new}), it follows that $\delta_y$ must lie in $\{0,1,\ldots, n_y\}$ and $\delta_y=n_y$ is only possible if $m_y\equiv 1\mod 4$. If $\delta_y=0$ or $\delta_y=n_y$ then $(D_{\lambda_y}-\lambda_y\mathbf{1}_{n_y})$ is the zero matrix, which implies, using Notation \ref{not:indecomposables}, that the $k$-span of $\mathcal{B}_{\lambda_y}$ gives a $kG$-module isomorphic to $N_{2,\lambda_y}^{\oplus n_y}$. Moreover, the statement of the theorem gives correctly $\ell_y=1$, $a_{y,1}=n_y$ and $a_{y,2}=0$ when $\delta_y\in\{0,n_y\}$. Suppose now that $0<\delta_y< n_y$. Then $(D_{\lambda_y}-\lambda_y\mathbf{1}_{n_y})$ is a non-zero nilpotent matrix that is upper triangular. Furthermore, its entries along the diagonal strips $(i,d+i)$ are all zero, for $0\le d <\delta_y$ and $1\le i \le n_y-d$, and its entries along the diagonal strip $(i,\delta_y+i)$ all equal $b_{y,\delta_y}\ne 0$, for $1\le i \le n_y-\delta_y$. Therefore, for $i\ge 0$, the rank of the matrix $(D_{\lambda_y}-\lambda_y\mathbf{1}_{n_y})^i$ equals $n_y-i\,\delta_y$ if $i\,\delta_y < n_y$ and $0$ otherwise. In particular, the definitions of $\ell_y$ and $a_{y,1}$ in part (ii) of the statement of the theorem imply that $(D_{\lambda_y}-\lambda_y\mathbf{1}_{n_y})^{\ell_y}= 0$ and $(D_{\lambda_y}-\lambda_y\mathbf{1}_{n_y})^{\ell_y-1}\ne 0$. We conclude that $D_{\lambda_y}$ is equivalent to a block diagonal matrix with $a_{y,1}$ Jordan blocks with eigenvalue $\lambda_y$ of size $\ell_y$ and $\delta_y-a_{y,1}$ Jordan blocks with eigenvalue $\lambda_y$ of size $\ell_y-1$. Using Notation \ref{not:indecomposables}, it follows that the $k$-span of $\mathcal{B}_{\lambda_y}$ gives a $kG$-module isomorphic to $N_{2\ell_y,\lambda_y}^{\oplus a_{y,1}} \oplus N_{2(\ell_y-1),\lambda_y}^{\oplus a_{y,2}}$.

Suppose next that we are in part (iii) of the theorem; in other words, we are in the situation of Remark \ref{rem:twocases}(ii). Then $m_y < M_y$ and $\nu_y\ge 1$. This means that $C_{\lambda_y}$ is a non-zero nilpotent upper triangular matrix whose only non-zero entries are equal to $1$ and occur precisely along the diagonal strip $(i,\nu_y+i)$, for $1\le i \le n_y-\nu_y$. On the other hand, $D_{\lambda_y}$ is an upper triangular matrix with non-zero diagonal entries. Define $\widetilde{C}_{\lambda_y} := D_{\lambda_y}^{-1}C_{\lambda_y}$. Then $\widetilde{C}_{\lambda_y}$ has $\nu_y$ zero columns, followed by the first $n_y-\nu_y$ columns of the upper triangular invertible matrix $D_{\lambda_y}^{-1}$. Therefore, for $i\ge 0$, the rank of the matrix $(\widetilde{C}_{\lambda_y})^i$ equals $n_y-i\,\nu_y$ if $i\,\nu_y < n_y$ and $0$ otherwise. In particular, the definitions of $\ell_y$ and $a_{y,1}$ in part (iii) of the statement of the theorem imply that $(\widetilde{C}_{\lambda_y})^{\ell_y}=0$ and $(\widetilde{C}_{\lambda_y})^{\ell_y-1}\ne 0$. We conclude that $\widetilde{C}_{\lambda_y}$ is equivalent to a block diagonal matrix with  $a_{y,1}$ Jordan blocks with eigenvalue $0$ of size $\ell_y$ and $\nu_y-a_{y,1}$ Jordan blocks with eigenvalue $0$ of size $\ell_y-1$. Using Notation \ref{not:indecomposables} together with Remark \ref{rem:01infty}, it follows that the $k$-span of $\mathcal{B}_{\lambda_y}$ gives a $kG$-module isomorphic to $N_{2\ell_y,\lambda_y}^{\oplus a_{y,1}} \oplus N_{2(\ell_y-1),\lambda_y}^{\oplus a_{y,2}}$. This completes the proof of Theorem \ref{thm:P1main}.
\end{proof}

\begin{remark}
\label{rem:smallEll}
Let $y\in Y_{\mathrm{br}}$. Then we have $\ell_y=1$ in Theorem \ref{thm:P1main} in the following cases:
\begin{itemize}
\item[(i)] If $\delta_y=0$ then always $\ell_y=1$.
\item[(ii)] If $\delta_y\ge 1$ then $\ell_y=1$ if and only if $m_y\equiv 1\mod 4$ and $\delta_y=\left\lfloor\frac{m_y}{4}\right\rfloor$. This follows from equation (\ref{eq:duh1new}) since $\delta_y \le \left\lfloor\frac{m_y}{4}\right\rfloor$.
\item[(iii)] If $\delta_y=-1$ then $\ell_y=1$ if and only if $m_y=1$. This follows since $\left\lfloor \frac{2m_y+3}{4}\right\rfloor - \left\lfloor \frac{m_y+3}{4}\right\rfloor\ge 0$.
\end{itemize}
In particular, these cases coincide, for $Y=\mathbb{P}^1_k$, with the cases covered by Theorem \ref{thm:onlyB1orB2B3small}.
\end{remark}

The next example shows that the list of isomorphism classes of indecomposable $kG$-modules that actually occur as direct summands of $\HH^0(X,\Omega_{X/k})$ in the situation of Theorem \ref{thm:P1main} for various $X$ is infinite and given as follows:
\begin{equation}
\label{eq:list2}
\left\{N_{2d,\lambda}\;:\;d \in\mathbb{Z}^+, \lambda\in k \cup\{\infty\}\right\} \cup \left\{M_{3,1},k\right\}.
\end{equation}
In particular, this list contains indecomposable $kG$-modules of arbitrarily large $k$-dimension.

\begin{example}
\label{ex:KatzGabber}
Let $Y=\mathbb{P}^1_k$ with function field $k(Y)=k(t)$. Given a positive integer $d$, we provide a family of examples such that there are at least two non-isomorphic indecomposable $kG$-modules whose $k$-dimensions are $2d$ and $2(d+1)$, respectively, and that occur as direct summands of $\HH^0(X,\Omega_{X/k})$. All these examples are Harbater-Katz-Gabber $G$-covers, in the sense of \cite[Section 4.B]{BleherChinburgPoonenSymonds2017}, i.e. there is exactly one branch point and this unique branch point is totally ramified. In the examples below, we have chosen the branch point to be at $\infty$. However, if we change the variable $t$ to $t':=t^{-1}$, we obtain examples of covers $X'$ of $\mathbb{P}^1_k$ such that the places in $k(X')$ above $\infty$ are unramified. 
In particular, $k(X')$ is then a global standard function field, in the sense of \cite[Definitions 2.1 and 3.1]{MarquesWard2018}. Since $\HH^0(X',\Omega_{X'/k})\cong \HH^0(X,\Omega_{X/k})$ as $kG$-modules, we obtain a contradiction to \cite[Theorem 6.4]{MarquesWard2018}, similarly to Examples \ref{ex:oops1} and \ref{ex:oops2}.

To cover all modules given in the list (\ref{eq:list2}), we need to consider two main cases. Fix $d\in\mathbb{Z}^+$ and $\lambda_0\in k-\{0,1\}$.

\begin{itemize}
\item[(1)] Let $p,q$ in Assumption \ref{ass:overP1} be given by
\begin{eqnarray*}
p&=&t^{8d+3},\\
q&=&\lambda_0^2 \,t^{8d+3} (1+t^{-4}),
\end{eqnarray*}
which implies that $r=p+q = (1+\lambda_0^2+\lambda_0^2t^{-4})\,t^{8d+3}$. Then
$$\mathrm{ord}_{\infty}(p) = -(8d+3) = \mathrm{ord}_{\infty}(q) = \mathrm{ord}_{\infty}(r),$$
and $\mathrm{ord}_y(p),\mathrm{ord}_y(q),\mathrm{ord}_y(r)\ge 0$ for all other $y\in\mathbb{P}^1_k$. This means that $Y_{\mathrm{br}}=\{\infty\}$, and $m_{\infty}=8d+3=M_{\infty}$. If $w:=v+  \lambda_0 (1+t^{-2})\,y$ then 
$\mathrm{ord}_{x_\infty}(w) = -(8d+3)$
when $x_\infty\in X$ lies above $\infty$. 
Using Remark \ref{rem:twocases} and equation (\ref{eq:uniformizer}), it follows that $\beta_{\infty} = \lambda_0 (1+\pi_{\infty}^2)$. Therefore, in Notation \ref{not:twocases}, we have $\lambda_{\infty}=\lambda_0$ and $\delta_{\infty}=2$. Since
$$\textstyle\left\lfloor \frac{2m_{\infty}+3}{4}\right\rfloor - \left\lfloor \frac{m_{\infty}+3}{4}\right\rfloor = 4d+2-(2d+1)=2d+1,$$
we obtain by part (ii) of Theorem \ref{thm:P1main} that $\ell_{\infty}=d+1$ and $a_{\infty,1}=1=a_{\infty,2}$. Therefore,
$$\HH^0(X,\Omega_{X/k})\cong  N_{2(d+1),\lambda_0} \;\oplus\; N_{2d,\lambda_0} \;\oplus\; M_{3,1}^{\oplus (2d)} \;\oplus\; k^{\oplus (2d+1)}.$$

\item[(2)] Define
$$f=t^{8d-5}, \quad  g=\lambda_0^2\,t^{8d-1}, \quad \mbox{and} \quad h=t^{8d-1}(\lambda_0^2+t^{-4})$$
which implies that the sum of any two of these polynomials equals the third one. Then
$$\mathrm{ord}_{\infty}(f) = -(8d-5) \quad\mbox{and}\quad  \mathrm{ord}_{\infty}(g) = -(8d-1) = \mathrm{ord}_{\infty}(h),$$
and $\mathrm{ord}_y(f),\mathrm{ord}_y(g),\mathrm{ord}_y(h)\ge 0$ for all other $y\in\mathbb{P}^1_k$.

\begin{itemize}
\item[(a)] Let $p,q$ in Assumption \ref{ass:overP1} be given by $p=f$ and $q=g$. Then $Y_{\mathrm{br}}=\{\infty\}$, and $m_{\infty}=8d-5$ and $M_{\infty}=8d-1$. If $w:=v+  \lambda_0 t^2u$ then 
$\mathrm{ord}_{x_\infty}(w) =-(8d+3)$
when $x_\infty\in X$ lies above $\infty$. Using Remark \ref{rem:twocases}, it follows that $\beta_{\infty} = \lambda_0 \,\pi_{\infty}^{-2}$. Therefore, in Notation \ref{not:twocases}, we have $\delta_{\infty}=-1$ and $\lambda_{\infty}=\infty$. Since
$$\textstyle \left\lfloor \frac{2m_{\infty}+3}{4}\right\rfloor - \left\lfloor \frac{m_{\infty}+3}{4}\right\rfloor +\frac{M_{\infty}-m_{\infty}}{2} = 4d-2-(2d-1) + 2 =2d+1,$$
we obtain by part (iii) of Theorem \ref{thm:P1main} that $\ell_{\infty}=d+1$ and $a_{y_\infty,1}=1=a_{y_\infty,2}$. Therefore,
$$\HH^0(X,\Omega_{X/k})\cong  N_{2(d+1),\infty} \;\oplus\; N_{2d,\infty} \;\oplus\; M_{3,1}^{\oplus (2d-2)} \;\oplus\; k^{\oplus (2d-1)}.$$

\item[(b)] If we define $p,q$ in Assumption \ref{ass:overP1} by $p=g$ and $q=f$, then everything stays the same as in part (a) except that we now have $w:=u+\lambda_0 t^2v$ and $\lambda_{\infty}=0$. Hence we obtain
$$\HH^0(X,\Omega_{X/k})\cong  N_{2(d+1),0} \;\oplus\; N_{2d,0} \;\oplus\; M_{3,1}^{\oplus (2d-2)} \;\oplus\; k^{\oplus (2d-1)}.$$

\item[(c)] If we define $p,q$ in Assumption \ref{ass:overP1} by $p=h$ and $q=g$, then everything stays the same as in part (a) except that we now have $w:=v + \lambda_0 t^2(u+v)$ and $\lambda_{\infty}=1$. Hence we obtain
$$\HH^0(X,\Omega_{X/k})\cong  N_{2(d+1),1} \;\oplus\; N_{2d,1} \;\oplus\; M_{3,1}^{\oplus (2d-2)} \;\oplus\; k^{\oplus (2d-1)}.$$
\end{itemize}
\end{itemize}
\end{example}

Theorem \ref{thm:awkward2} follows from Theorem \ref{thm:P1main} and Example \ref{ex:KatzGabber}.

\begin{remark}
\label{rem:sodumber!}
It follows from Remark \ref{rem:candomore} and Example \ref{ex:KatzGabber} that by adding $kG$ to the list (\ref{eq:list2}), we obtain a complete list of isomorphism classes of indecomposable $kG$-modules that actually occur as direct summands of $\HH^0(X,\Omega_{X/k})$ as soon as the inequality in (\ref{eq:sodumb!}) is satisfied, with no restriction on the genus of $Y=X/G$.
\end{remark}


\section{Appendix: Characteristic two Klein four representation theory}
\label{s:klein4rep}

Let $k$ be an algebraically closed field of characteristic 2 and let $G=\langle \sigma,\tau\rangle$ be isomorphic to $\mathbb{Z}/2\times \mathbb{Z}/2$. 

\begin{nota}
\label{not:indecomposables}
We define the following infinite list of $kG$-modules:
\begin{itemize}
\item[(a)] Let $k$ denote the trivial simple $kG$-module, and let $kG$ denote the free rank one $kG$-module.

\item[(b)] For every positive integer $n$ and every $\lambda\in k\cup\{\infty\}$, let $N_{2n,\lambda}$ be the $kG$-module of $k$-dimension $2n$ such that $\sigma-1$ and $\tau-1$ act as the following matrices with respect to a suitable $k$-basis of $N_{2n,\lambda}$:
\begin{itemize}
\item[$\bullet$] if $\lambda\in k$ then
$$\sigma-1 \;\longleftrightarrow\; \left(\begin{array}{c|c} \mathbf{0}_n&\mathbf{1}_n\\ \hline \mathbf{0}_n&\mathbf{0}_n\end{array}\right)
\qquad\mbox{and}\qquad
\tau-1\;\longleftrightarrow\; \left(\begin{array}{c|c} \mathbf{0}_n&J_n(\lambda)\\ \hline \mathbf{0}_n&\mathbf{0}_n\end{array}\right)  ,$$
\item[$\bullet$] and if $\lambda = \infty$ then
$$\sigma-1 \;\longleftrightarrow\; \left(\begin{array}{c|c} \mathbf{0}_n&J_n(0)\\ \hline \mathbf{0}_n&\mathbf{0}_n\end{array}\right) 
\qquad\mbox{and}\qquad
\tau-1\;\longleftrightarrow\; \left(\begin{array}{c|c} \mathbf{0}_n&\mathbf{1}_n\\ \hline \mathbf{0}_n&\mathbf{0}_n\end{array}\right),$$
\end{itemize}
where $\mathbf{0}_n$ (resp. $\mathbf{1}_n$) is the $n\times n$ zero matrix (resp. the $n\times n$ identity matrix), and $J_n(\lambda)$ is the upper triangular Jordan block of size $n$ with eigenvalue $\lambda$.

\item[(c)] For every positive integer $n$, let $M_{2n+1,1}$ and $M_{2n+1,2}$ be two $kG$-modules of $k$-dimension $2n+1$ such that $\sigma-1$ and $\tau-1$ act as the following matrices with respect to a suitable $k$-basis of $M\in\{ M_{2n+1,1} , M_{2n+1,2}\}$:
\begin{itemize}
\item[$\bullet$] if $M=M_{2n+1,1}$ then
$$\sigma-1 \;\longleftrightarrow\; \left(\begin{array}{c|cc} \mathbf{0}_n&\mathbf{1}_n&\mathbf{0}_{n\times 1}\\ \hline \mathbf{0}_n&\mathbf{0}_n&\mathbf{0}_{n\times 1}\\ \mathbf{0}_{1\times n}&\mathbf{0}_{1\times n} & 0\end{array}\right)
\qquad\mbox{and}\qquad
\tau-1\;\longleftrightarrow\; \left(\begin{array}{c|cc} \mathbf{0}_n&\mathbf{0}_{n\times 1}&\mathbf{1}_n\\ \hline \mathbf{0}_n&\mathbf{0}_{n\times 1}&\mathbf{0}_n\\ \mathbf{0}_{1\times n} & 0&\mathbf{0}_{1\times n}\end{array}\right) ,$$
\item[$\bullet$] and if $M=M_{2n+1,2}$ then
$$\sigma-1 \;\longleftrightarrow\; \left(\begin{array}{cc|c} \mathbf{0}_n&\mathbf{0}_{n\times 1}&\mathbf{1}_n\\ \mathbf{0}_{1\times n}&0&\mathbf{0}_{1\times n}\\\hline \mathbf{0}_n&\mathbf{0}_{n\times 1}&\mathbf{0}_n\end{array}\right)
\qquad\mbox{and}\qquad
\tau-1\;\longleftrightarrow\; \left(\begin{array}{cc|c} \mathbf{0}_n&\mathbf{0}_{n\times 1}&\mathbf{0}_{1\times n}\\ \mathbf{0}_{1\times n}&0&\mathbf{1}_n\\\hline \mathbf{0}_n&\mathbf{0}_{n\times 1}&\mathbf{0}_n\end{array}\right) ,$$
\end{itemize}
where $\mathbf{0}_{n\times 1}$ (resp. $\mathbf{0}_{1\times n}$) is the zero column vector (resp. zero row vector) of length $n$.
\end{itemize}
\end{nota}

\begin{remark}
\label{rem:indecomposables}
Each indecomposable $kG$-module is isomorphic to one of the modules in the list
$$\left\{k,kG\right\} \cup
\left\{N_{2n,\lambda}\;:\;n \in\mathbb{Z}^+, \lambda\in k \cup\{\infty\}\right\} \cup 
\left\{M_{2n+1,1},M_{2n+1,2}\;:\;n \in\mathbb{Z}^+\right\}$$
and no two such modules are isomorphic to each other (see, for example, \cite{Basev1961} or \cite{Conlon1965}). 
\end{remark}

\begin{remark}
\label{rem:01infty}
For each positive integer $n$, there is the following connection between $N_{2n,\infty}$, $N_{2n,0}$ and $N_{2n,1}$. Let $(\rho_1,\rho_2)$ be an ordered pair of two distinct elements of $\{\sigma,\tau,\sigma\circ\tau\}$, i.e. $(\rho_1,\rho_2)$ is an ordered pair of generators of $G$. Consider the $kG$-module $M$ of $k$-dimension $2n$ with actions by $\rho_1-1$ and $\rho_2-1$ given as the following matrices with respect to a suitable $k$-basis of $M$:
$$\rho_1-1 \;\longleftrightarrow\; \left(\begin{array}{c|c} \mathbf{0}_n&J_n(0)\\ \hline \mathbf{0}_n&\mathbf{0}_n\end{array}\right) 
\qquad\mbox{and}\qquad
\rho_2-1\;\longleftrightarrow\; \left(\begin{array}{c|c} \mathbf{0}_n&\mathbf{1}_n\\ \hline \mathbf{0}_n&\mathbf{0}_n\end{array}\right),$$
Since the product of these two matrices is the zero matrix, it follows that $\rho_1\circ\rho_2-1$ acts as 
$$\rho_1\circ\rho_2-1\;\longleftrightarrow\; \left(\begin{array}{c|c} \mathbf{0}_n&J_n(1)\\ \hline \mathbf{0}_n&\mathbf{0}_n\end{array}\right).$$
We obtain:
\begin{itemize}
\item[(a)] If $(\rho_1,\rho_2)=(\sigma,\tau)$ then $M\cong N_{2n,\infty}$.
\item[(b)] If $(\rho_1,\rho_2)=(\tau,\sigma)$ then $M\cong N_{2n,0}$.
\item[(c)] If $(\rho_1,\rho_2)=(\sigma\circ\tau,\tau)$ then $M\cong N_{2n,1}$.
\end{itemize}
\end{remark}

\begin{remark}
\label{rem:cohorts}
A similar phenomenon to the one described in Remark \ref{rem:01infty} occurs for all parameters in $k\cup \{\infty\}$. More precisely, fix a positive integer $n$ and fix $\lambda \in k\cup \{\infty\}$. For any permutation $\xi$ of the set $\{\sigma,\tau,\sigma\circ\tau\}$, define $N_{2n,\lambda}^\xi$ to be the $kG$-module with the same underlying $k$-vector space as $N_{2n,\lambda}$ but on which each non-identity element $g\in G$ acts as $\xi(g)$. Let $\xi_1$ be the permutation interchanging $\sigma$ and $\tau$, and let $\xi_2$ be the permutation interchanging $\tau$ and $\sigma\circ \tau$. Identifying $\frac{1}{0}=\infty$, $\frac{1}{\infty}=0$ and $1+\infty = \infty$, we obtain
$$N_{2n,\lambda}^{\xi_1} \cong N_{2n,\frac{1}{\lambda}}\quad\mbox{and}\quad
N_{2n,\lambda}^{\xi_2} \cong N_{2n,1+\lambda}.$$
Let $S_\lambda$ be the set of all $\mu\in k\cup\{\infty\}$ for which there exists a permutation $\xi$ of $\{\sigma,\tau,\sigma\circ\tau\}$ such that $N_{2n,\lambda}^{\xi}\cong N_{2n,\mu}$. Then 
$$S_\lambda=\left\{\lambda,\frac{1}{\lambda},1+\lambda,\frac{1}{1+\lambda},\frac{\lambda}{1+\lambda},\frac{1+\lambda}{\lambda}\right\}.$$
This set contains precisely 6 elements for all $\lambda\in k-\mathbb{F}_4$. Moreover, 
$$S_0=\{0,1,\infty\} = S_1 = S_\infty,$$
and, if $\mathbb{F}_4^\times = \langle \alpha \rangle$, then 
$$S_{\alpha}=\{\alpha,1+\alpha\} = S_{1+\alpha}.$$
\end{remark}

Fix $(a:b)\in \mathbb{P}^1_k$. Table \ref{tab:succdims} provides, for every indecomposable $kG$-module $U$, the $k$-dimensions of the subquotients $U^{(i+1)}/U^{(i)}$ for $0\le i \le 3$, as defined in Notation \ref{not:operators}.
\begin{table}[ht]
\caption{Dimensions of successive subquotients of indecomposable $kG$-modules.}
\label{tab:succdims}
$\begin{array}{c||c|c|c|c}
U&\mathrm{dim}_k\,U^{(4)}/U^{(3)}&\mathrm{dim}_k\,U^{(3)}/U^{(2)}&\mathrm{dim}_k\,U^{(2)}/U^{(1)}&\mathrm{dim}_k\,U^{(1)}/U^{(0)} \\ \hline\hline
k&0&0&0&1\\ \hline
kG&1&1&1&1 \\ \hline
\begin{array}{c} N_{2n,\lambda}\\\mbox{\footnotesize $(\lambda\in k\cup\{\infty\})$}\end{array}&0&\begin{array}{cl} n-1 & \mbox{if $\lambda =a/b$}\\
n & \mbox{if $\lambda\ne a/b$}\end{array} & \begin{array}{cl} 1 & \mbox{if $\lambda = a/b$}\\
0 & \mbox{if $\lambda \ne a/b$}\end{array} & n \\ \hline
M_{2n+1,1}&0&n&1&n\\ \hline
M_{2n+1,2}&0&n&0&n+1
\end{array}$
\end{table}

\bibliographystyle{plain}

\end{document}